\documentclass[a4paper,oneside,12pt,reqno,final]{amsart}

\usepackage{amsfonts,amssymb,amsmath,amsthm}
\usepackage[headinclude,DIV13]{typearea}
\usepackage{hyperref}
\usepackage{mathrsfs,bbm,stmaryrd}
\usepackage[color]{showkeys}
\definecolor{refkey}{rgb}{1,0,0} 
\definecolor{labelkey}{rgb}{0,0,1}

\usepackage{graphicx, psfrag}

\newcommand{\al}{\alpha}

\newcommand{\g}{\gamma}
\newcommand{\de}{\delta}
\newcommand{\e}{\varepsilon}

\newcommand{\tht}{\theta}

\newcommand{\ka}{\kappa}

\newcommand{\la}{\lambda}

\newcommand{\n}{\nu}

\newcommand{\s}{\sigma}

\newcommand{\f}{\phi}

\newcommand{\om}{\omega}

\newcommand{\G}{\Gamma}
\newcommand{\D}{\Delta}

\newcommand{\Ps}{\Psi}
\newcommand{\Om}{\Omega}

\newcommand{\bbN}{\mathbb N}

\newcommand{\bbR}{\mathbb R} 

\newcommand{\bbP}{\mathbb P}
\newcommand{\bbE}{\mathbb E}

\newcommand{\1}{\mathbbm{1}}

\newcommand{\cal}{\mathcal}

\newcommand{\cB}{\cal B}
\newcommand{\cC}{\cal C}
\newcommand{\cD}{\cal D}
\newcommand{\cE}{\cal E}
\newcommand{\cF}{\cal F}

\newcommand{\cH}{\cal H}

\newcommand{\cM}{\cal M}
\newcommand{\cN}{\cal N}

\newcommand{\calc}{\mathscr}

\newcommand{\ccE}{\calc E}
\newcommand{\ccF}{\calc F}

\newcommand{\ccN}{\calc N}

\newcommand{\ccP}{\calc P}

\newcommand{\ccS}{\calc S}

\newcommand{\wt}{\widetilde}
\newcommand{\wh}{\widehat}

\newtheorem{theorem}{Theorem}[section]
\newtheorem{lemma}[theorem]{Lemma}
\newtheorem{proposition}[theorem]{Proposition}
\newtheorem{corollary}[theorem]{Corollary}

\theoremstyle{definition}
\newtheorem{definition}[theorem]{Definition}
\newtheorem{assumption}[theorem]{Assumptions}

\theoremstyle{remark}
\newtheorem{remark}{Remark}

\newcommand{\be}{\begin{equation}}
\newcommand{\ee}{\end{equation}}

\newcommand{\bea}{\be\begin{aligned}}
\newcommand{\eea}{\end{aligned}\ee}

\newcommand{\bal}{\begin{aligned}}
\newcommand{\eal}{\end{aligned}}

\newcommand \ba {\begin{array}}
\newcommand \ea {\end{array}}

\newcommand{\abs}[1]{\left\lvert{#1}\right\rvert}

\newcommand{\norm}[1]{\left\lVert{#1}\right\rVert}

\newcommand{\bra}[1]{\left\langle{#1}\right\rangle}
\newcommand{\floor}[1]{\left\lfloor{#1}\right\rfloor}
\newcommand{\pare}[1]{\left({#1}\right)}
\newcommand{\cro}[1]{\left[{#1}\right]}
\newcommand{\acc}[1]{\left\{{#1}\right\}}

\newcommand{\lawconv}[2]{\xrightarrow[#1\to#2]{\mathscr L}}
\newcommand{\conv}[2]{\xrightarrow[#1\to#2]{}}

\newcommand\frechet[3]{\frac{\de^{#3}#1}{{\de #2}^{#3}}}
\newcommand\grad{\nabla}
\newcommand\del{\partial}
\newcommand\dint{\mathrm{d}}

\newcommand\dx{\dint x}
\newcommand\dy{\dint y}
\newcommand\ds{\dint s}
\newcommand\dt{\dint t}

\newcommand\argmax{\mathop{\mathrm{argmax}}}

\newcommand\capa[1]{\mathrm{cap}\pare{#1}}

\newcommand\Det{\mathrm{Det}}

\newcommand{\seg}{\geqslant}
\newcommand{\ieg}{\leqslant}

\numberwithin{equation}{section}

\begin{document}

\title[Sharp metastable asymptotics for one dimensional SPDEs]{
Sharp asymptotics of metastable transition times for one dimensional SPDEs
}
\author{Florent Barret}
\address{CMAP UMR 7641, \'Ecole Polytechnique CNRS, Route de Saclay,
91128 Palaiseau Cedex, France}
\email{barret@cmap.polytechnique.fr}

\begin{abstract}
We consider a class of parabolic semi-linear stochastic partial differential equations driven by space-time white noise on a compact space interval. 
Our aim is to obtain precise asymptotics of the transition times between metastable states. A version of the so-called Eyring-Kramers Formula is proven in an infinite dimensional setting. The proof is based on a spatial finite difference discretization of the stochastic partial differential equation. The expected transition time is computed for the finite dimensional approximation and controlled uniformly in the dimension. 
\end{abstract}

\subjclass[2010]{82C44; 60H15, 35K57.}
\keywords{Metastability, metastable transition time, parabolic stochastic partial differential equations, reaction-diffusion equations, stochastic Allen-Cahn equations, Eyring-Kramers formula.}
\date{\today}

\maketitle

\section{Introduction}\label{sec:intro}

Metastability is a phenomenon which concerns systems with several stable states. Due to perturbations (either deterministic or stochastic) the system undergoes a shift of regime and reaches a new stable state (see e.g. \cite{cgov84} by Cassandro, Galves, Olivieri and Vares, the book \cite{olivieri.vares05} by Olivieri and Vares and the lecture notes \cite{bovier09} by Bovier). Typical examples of metastable behavior can be found in chemistry, physics (for models of phase transition) and ecology. 

In this article, our aim is to understand metastability for a class of stochastic partial differential equations. We consider the Allen-Cahn (or Ginzburg-Landau) model which represents the behavior of an elastic string in a viscous stochastic environment submitted to a potential (see e.g. Funaki \cite{funaki83}). This model has other interpretations in quantum field theory (see \cite{faris.jonalasinio82, cop86} and the references therein) and in statistical mechanics as a reaction diffusion equation modeling phase transitions and evolution of interfaces (see Brassesco and Butt\`a \cite{brassesco91, brassesco.butta98}).
 
More precisely, we deal with the following equation, for $(x,t)\in[0,1]\times\bbR^+$
\be\label{eq:spde.1}
\del_t u(x,t)=\g\del_{xx}u(x,t)-V'(u(x,t))+\sqrt{2\e}W
\ee
where $\g>0$. $W$ is a space-time white noise on $[0,1]\times\bbR^+$ in the sense of Walsh \cite{walsh} and $\e>0$ is the intensity of the noise. $V$ is a smooth real valued function on $\bbR$ called a local potential.
We consider two boundary conditions: Dirichlet boundary conditions (for all $t\in\bbR^+$, $u(0,t)=u(1,t)=0$) and Neumann boundary conditions ($\del_xu(0,t)=\del_xu(1,t)=0$).
The initial condition is given by a continuous function $u_0$ which satisfies the given boundary conditions. Existence and uniqueness of an H\"older-continuous solution in the mild sense have been proved by Gy\"ongy and Pardoux in \cite{gyongy.pardoux93}.

Faris and Jona-Lasinio in \cite{faris.jonalasinio82} are among the first ones to analyze Equation \eqref{eq:spde.1} for a double well potential
\be\label{eq:allen.cahn.1}
V(x)=\frac{x^4}4-\frac{x^2}2.
\ee
In this case, $V$ has only two minima which are $+1$ and $-1$. One expects that the model \eqref{eq:spde.1} has several stable states and that a metastable behavior occurs.
The authors introduced a functional potential $S$ and interpreted \eqref{eq:spde.1} as the stochastic perturbation of an infinite dimensional gradient system:
\be\label{eq:spde.2}
\del_tu=-\frechet{S}{\phi}{}+\sqrt{2\e}W
\ee
where for $\phi$ a differentiable function,
\be\label{eq:potential.1}
S(\phi)=\int_0^1\frac{\g}2\abs{\phi'(x)}^2+V(\phi(x))\dx.
\ee
$S$ represents the free energy. $\frechet{S}{\phi}{}$ is the Fr\'echet derivative of $S$ i.e. the infinite dimensional gradient of $S$.

For more general functions $V$ (real valued $C^3$ functions), we can define a similar potential $S$ as in \eqref{eq:potential.1} which determines a potential landscape. Under the stochastic perturbation, this potential landscape is explored by the process $u$ defined in \eqref{eq:spde.1}. 
While the system without noise (i.e. $\e=0$) has several stable fixpoints (which are the minima of $S$), for $\e>0$ transitions between these fixpoints will occur at a suitable timescale. The transition paths go through the lowest saddle points.  Thus, minima and saddle points of $S$ have a key role to understand metastability but it is often a hard task, given a potential $V$ (and thus $S$), to completely compute and comprehend the geometrical structure of the energy landscape. However, some elegant method exists (see e.g. \cite{fiedler.rocha96, wolfrum02}).

The model \eqref{eq:spde.2} is an infinite dimensional generalization of the finite dimensional systems investigated by Freidlin and Wentzell \cite{freidlin.wentzell84} and by Bovier, Eckhoff, Gayrard and Klein in \cite{bovier04, bovier05}. Moreover, we will see that \eqref{eq:spde.1} is rigorously the limit of a gradient finite dimensional system (via a spatial finite difference approximation).

\bigskip

Our aim is to derive precise asymptotics of the expected transition time i.e. the time needed, starting from a minimum $\phi_0$ of $S$, to hit a set of lower minima. 
We define the hitting time $\tau_\e(B)$ by 
$\tau_\e(B)=\inf\acc{t>0,u(t)\in B}$
where $B$ is a disjoint union of small ball around some minima of $S$ lower than $\phi_0$. We prove that the expected time, $\bbE_{\phi_0}[\tau_{\e}(B)]$, has a very distinctive form known as the Arrhenius equation (Theorem \ref{th:main}). This expectation reads
\be\label{eq:arrhenius}
\bbE_{\phi_0}[\tau_{\e}(B)]= Ae^{E/\e}(1+O(\sqrt{\e}\abs{\ln(\e)}^{3/2}))\quad(\e\to0)
\ee
where $E$ is the activation energy and $A$ is the prefactor. $E$ has been computed by Faris and Jona-Lasinio for the double well potential \eqref{eq:allen.cahn.1} using a large deviation approach (Theorem 1.1 \cite{faris.jonalasinio82}). $E$ is exactly the minimum height of potential that a pathway has to overcome to reach $B$ starting from $\phi_0$. The prefactor $A$ is a constant (for our set of hypotheses) and depends only on the local geometry of the potential $S$ near the minimum $\phi_0$ and near the passes (or saddle points) from $\phi_0$ to the set $B$. The order $O(\sqrt{\e}\abs{\ln(\e)}^{3/2})$ of the error term comes directly from the local approximation of the potential $S$ by its quadratic part. 

For the double well potential \eqref{eq:allen.cahn.1} with Neumann boundary conditions, Faris and Jona-Lasinio proved that $S$ has only two global minima, denoted $m$ and $-m$ (corresponding roughly to the constant functions $1$ and $-1$ resp.). For some $\g$, this model has a unique saddle point $\s=0$ (the constant function $0$). We deduce from Theorem \ref{th:main} that $\bbE_{-m}[\tau_{\e}(B^+)]$, for a small ball $B^+$ in the suitable norm around $m$, takes the form \eqref{eq:arrhenius} with $E=S(\s)-S(-m)$ and 
\be\label{eq:eyring-kramers.0}
A=\frac{2\pi}{\abs{\la^-(\s)}} \sqrt{\prod_{k=1}^{+\infty}\frac{\abs{\la_k(\s)}}{\abs{\la_k(-m)}}}
\ee
where $(\la_k(\phi))_{k\seg 1}$ are the eigenvalues of the second Fr\'echet derivative of the potential $S$ at a point $\phi$ and $\la^-(\s)$ is the unique negative eigenvalue at the saddle point $\s$. Using asymptotic expansion of the eigenvalues, we prove that the infinite product converges. It is exactly the equivalent for an operator of the classical determinant of a matrix. We also mention the fact that this infinite product has a nice expression in terms of solutions of linear differential equations (see e.g. Levit and Smilansky \cite{levit.smilansky77}).  

\bigskip

Eyring in \cite{eyring35} and particularly Kramers in \cite{kramers40} investigate the case of a one dimensional diffusion as a model for chemical reaction rates and express rates instead of expectations. Their formula is known as the Eyring-Kramers Formula. It takes the form \eqref{eq:arrhenius} with the prefactor given by  a formula similar to \eqref{eq:eyring-kramers.0} but with a single factor in the product (there is only one eigenvalue). 

Similar Eyring-Kramers Formulas exist through a wide range of reversible Markovian models from Markov chains, stochastic differential equations. 
For finite dimensional diffusions, Freidlin and Wentzell in \cite{freidlin.wentzell84}, proving that these systems obey a large deviation principle, obtained the activation energy in terms of the rate function. In recent years, the potential theory approach initiated by Bovier, Eckhoff, Gayrard and Klein in \cite{bovier04, bovier05} has allowed to give very precise results and led to a proof of the Eyring-Kramers Formula for gradient drift diffusions in finite dimension. Moreover, the potential approach originate from Markov chains (see \cite{begk1, begk2, bovier09}) and have been refined to obtain metastable transition times for specific models (see e.g. \cite{BBI08, bdhs10}).

Formula \eqref{eq:eyring-kramers.0} is then the extension of the Eyring-Kramers Formula to a class of one-dimensional SPDEs \eqref{eq:spde.1}.
Maier and Stein in \cite{maier.stein01} obtained heuristically this formula and Vanden-Eijnden and Westdickenberg in \cite{vde.west08} used it to compute nucleation probability.
\bigskip

Specifically, the system \eqref{eq:spde.1} and its metastable behavior have been studied for at least thirty years using mainly large deviation principle and comparison estimates between the deterministic process (\eqref{eq:spde.1} with $\e=0$) and the stochastic process defined by \eqref{eq:spde.1}.
Cassandro, Olivieri, Picco \cite{cop86} obtained similar asymptotics as Faris and Jona-Lasinio \cite{faris.jonalasinio82} when the size of the space interval is not fixed and goes to infinity as $\e$ goes to $0$ sufficiently slowly. These results first prove the existence of a suitable exponential timescale in which the process undergoes a transition.

In the same case as \eqref{eq:allen.cahn.1}, Martinelli, Olivieri and Scoppola \cite{martinelli89} obtain the asymptotic exponentiality of the transition times (Theorem 4.1 \cite{martinelli89}).
Also, Brassesco \cite{brassesco91} proves that the trajectories of this system exhibit characteristics of a metastable behavior: the escape from the basin of attraction of the minimum $-m$ occurs through the lowest saddle points (Theorem 2.1 \cite{brassesco91}) and the process starting from $-m$ spends most of its time before the transition near $-m$ (Theorem 2.2 \cite{brassesco91}).

\bigskip

In this paper, we consider a local potential $V$ (satisfying Assumptions \ref{hyp:V} and \ref{hyp:S}) and we rigorously prove an infinite dimensional version of the Eyring-Kramers Formula.  Our method relies on a spatial finite difference approximation of Equation \eqref{eq:spde.1} introduced by Berglund, Fernandez and Gentz in \cite{berglund107, berglund207} as a model of coupled particles submitted to a potential. The computation of the expected transition time for the approximated system gives us the prefactor, the activation energy and some error terms. We need to control the behavior of these error terms as the step of discretization goes to $0$ (or equivalently as the dimension $N$ of the approximated system goes to $+\infty$). To this aim, we adapt results from \cite{bbm10} by Bovier, M\'el\'eard and the author.

As proved by Funaki \cite{funaki83} and Gy\"ongy \cite{gyongy98}, the solution of the approximated system converges to the solution of the SPDE.
By combining different results from SPDE theory, large deviation theory (from Chenal and Millet \cite{chenal.millet97}) and Sturm-Liouville theory we are able to take the limit of the finite dimensional model in order to retrieve the SPDE \eqref{eq:spde.1}. We also need to adapt estimates on the loss of the memory of the initial condition (from Martinelli, Scoppola and Sbano \cite{martinelli89,sbano94}) uniformly in the dimension.


%

The use of spatial finite difference approximation is quite natural since we consider our SPDEs in the sense of Walsh \cite{walsh}, limited to the case of space-time white noise. Other approximations could be possible, notably the Galerkin approximation should lead to similar results for a different class of SPDEs  in the framework of Da Prato and Zabczyk (see the book \cite{daprato.zabczyk}). 

\bigskip

The article is organized as follows. In Section \ref{sec:results}, we present the equation, the assumptions, the main theorem (Theorem \ref{th:main}) and a sketch of its proof. Then in Section \ref{sec:conv}, we adapt the convergence of the approximations and prove convergence of the approximated transition times. In Section \ref{sec:initial}, we state large deviations estimates by Chenal and Millet \cite{chenal.millet97}, contraction results by Martinelli, Olivieri, Scoppola and Sbano \cite{martinelli89, sbano94} and prove a uniform control in the initial condition uniformly in the dimension. In Section \ref{sec:prefactor}, we recall results about eigenvalues and eigenvectors of Sturm-Liouville problems and prove the convergence of the prefactor. In the last section, we compute  the expected transition times uniformly in the dimension.

We will use the following notations henceforth.
For a functional space $\cC$, equipped with a norm $\norm{\cdot}_{\cC}$, we denote by $\cC_{bc}$ the closed subspace in the $\cC$ topology of the functions in $\cC$ satisfying the suitable boundary conditions (Dirichlet or Neumann). 
For $f\in L^{\infty}([0,1]\times [0,T])$ we set the norm of this space $\norm{f}_{\infty,T}$ or simply $\norm{f}_{\infty}$ when $T=+\infty$.

\bigskip

\noindent\textbf{Acknowledgments.} I am very grateful to Anton Bovier and 
Sylvie M\'el\'eard for suggesting this  topic and for 
constant help and advice. 
I am indebted to the Hausdorff Center for Mathematics  Bonn for financial 
support of numerous visits to Bonn. Part of the work in this paper has been 
realized at the Technion in Haifa at the invitation of Dima Ioffe whom 
I thank for his kind hospitality. 
The research on this project was supported by ANR MANEGE.

\section{Results}\label{sec:results}

\subsection{The Equation}\label{ssec:equation}

The assumptions are of two kinds: some on the local potential $V$, others on the functional potential $S$. We first start with the hypotheses on $V$.
\begin{assumption}\label{hyp:V}
We suppose that:
\begin{itemize}
\item $V$ is $C^3$ on $\bbR$.
\item $V$ is convex at infinity: there exist $R,c>0$ such that for $\abs{u}>R$
\be\label{eq:V.1}
V''(u)>c>0.
\ee
\item $V$ grows at infinity at most polynomially: there exist $p,C>0$ such that
\be\label{eq:V.2}
V(u)<C(1+\abs{u}^p).
\ee
\end{itemize}
\end{assumption}
These hypotheses are made to avoid complications for the definition of the solution $u$ of \eqref{eq:spde.1} and to allow the computations of the derivatives of $S$.

Let $(\Om,\ccF,\bbP)$ be a probability space on which we define a space-time white noise $W$ as defined in \cite{walsh} equipped with a filtration $(\cF_t)_{t\seg0}$ with the usual properties. The integrable processes for the white noise are the predictable measurable processes in $L^2(\Om\times\bbR_+\times[0,1])$. We denote by $g_t(x,y)$ the density  of the semi-group generated by $\g\del_{xx}$ on $[0,1]$ with the suitable boundary conditions.

Let us recall that a random field $u$ is a mild solution of \eqref{eq:spde.1} if
\begin{enumerate}
\item $u$ is almost surely continuous on $[0,1]\times\bbR^+$ and predictable
\item for all $(x,t)\in[0,1]\times\bbR^+$
\begin{align}\label{eq:mild.3}\nonumber
u(x,t)&=\int_0^1g_t(x,y)u_0(y)\dy-\int_0^t\int_0^1g_{t-s}(x,y)V'(u(y,s))\dy\ds\\
&+\sqrt{2\e}\int_0^t\int_0^1g_{t-s}(x,y)W(\dy,\ds).
\end{align}
\end{enumerate}

We state from \cite{gyongy.pardoux93} the following result on the existence, uniqueness and regularity of the solution.
\begin{proposition}[\cite{gyongy.pardoux93}]\label{prop:sol}
For every initial condition $u_0\in C_{bc}([0,1])$, the stochastic partial differential equation \eqref{eq:spde.1} has a unique mild solution.
Moreover for all $T>0$ and $p\seg1$,
\be\label{eq:moment.1}
\bbE\left[\sup_{[0,T]\times[0,1]}\abs{u(x,t)}^p\right]\ieg C(T,p).
\ee
The random field $u$ is essentially $\frac12$-H\"older in space and $\frac14$-H\"older in time.
\end{proposition}

The only complication comes from the fact that $V'$ is not globally Lipschitz but prevents the process to go to infinity. From Assumptions \ref{hyp:V}, we have
\be\label{eq:drift.1}
-xV'(x)<C.
\ee
The proof of Proposition \ref{prop:sol} is standard and uses mainly estimates on the density $g_t(x,y)$.

\begin{remark}\label{rem:dim.one}
The definition of the stochastic convolution (the last expression of the right-hand side of \eqref{eq:mild.3}) requires the density of the semi-group to be in $L^2([0,1]\times[0,T])$ for every $T>0$. Unfortunately, that is only true in dimension one. For higher dimensions, the stochastic convolution does not define a classical function but a distribution in a Sobolev space of negative index \cite{walsh}.
\end{remark}

\subsection{Stationary Points}

As for the finite dimensional case, the minima and saddle points of $S$ play a crucial role. To this end, we first specify what is the "gradient" (or the Fr\'echet derivative) of the functional $S$. Let us recall that $S$ is defined, for $\phi\in H^1_{bc}$, by
\be
S(\phi)=\int_0^1\frac{\g}2\abs{\phi'(x)}^2+V(\phi(x))\dx.
\ee

For $\phi,h$ in $C_{bc}^2([0,1])$ we have a Taylor expansion of $S$ at the second order in $h$
\be\label{eq:potential.2}
S(\phi+h)=S(\phi)+D_{\phi}S(h)+\frac12D^2_{\phi}S(h,h)+O(\norm{h}^2_{C^2})
\ee
where $\norm{h}_{C^2}=\norm{h}_{\infty}+\norm{h'}_{\infty}+\norm{h''}_{\infty}$. By integration by parts we compute the differentials $D_{\phi}S$ and $D^2_{\phi}S$. The first order differential is a linear functional which takes the form
\be\label{eq:frechet.1}
D_{\phi}S(h)=\int_0^1[-\g \phi''(x)+V'(\phi(x))]h(x)\dx.
\ee
The Fr\'echet derivative is $\frechet{S}{\phi}{}=-\g \phi''(x)+V'(\phi(x))$. The second order derivative (the Hessian operator) takes the form
\be\label{eq:frechet.2}
D^2_{\phi}S(h,h)=\int_0^1h(x)[-\g h''(x)+V''(\phi(x))h(x)]\dx.
\ee
We denote by $\cH_{\phi}S$ the Hessian operator at $\phi$:
\be
\cH_{\phi}Sh(x)=-\g h''(x)+V''(\phi(x))h(x).
\ee
The Hessian operator is a Sturm-Liouville operator.

We say that $\phi$ is a \emph{stationary point} of $S$ if $\phi$ is solution of the non-linear differential equation
\be\label{eq:frechet.3}
\frechet{S}{\phi}{}=-\g \phi''+V'(\phi)=0.
\ee

Let us now fix two points $\phi,\psi\in C_{bc}([0,1])$ and define some quantities.
\be\label{eq:saddle.1b}
\G(\phi\to\psi)=\acc{f,f(0)=\phi,f(1)=\psi,f\in C([0,1],C_{bc}([0,1]))}
\ee
is the set of continuous paths from $\phi$ to $\psi$. For $f\in\G(\phi\to\psi)$, $\wh f$ denotes the set of maxima of the path $f$,
\be\label{eq:saddle.1c}
\wh f=\acc{f(t_0), t_0\in\argmax_{t\in[0,1]}S(f(t))}.
\ee
The saddle points are passes from a valley to another one. The definition uses this idea.

\begin{definition}[Saddles]\label{def:saddle}
For any $\phi,\psi\in C_{bc}([0,1])$, we define $\wh{S}(\phi,\psi)$, the minimum height needed to go from $\phi$ to $\psi$
\be\label{eq:saddle.1}
\wh{S}(\phi,\psi)
=\wh{S}(\psi,\phi)
=\inf\acc{S(\phi), \phi\in \wh f, f\in \G(\phi\to\psi)}.
\ee

For $\phi,\psi$ such that $\wh{S}(\phi,\psi)<\infty$, we denote $\ccS(\phi,\psi)$ the set of admissible saddles: the points which realize the maximum along a minimal pathway
\be\label{eq:saddle.2}
\ccS(\phi,\psi)=\acc{\s\in C_{bc}([0,1]),S(\s)=\wh{S}(\phi,\psi), \exists f\in\G(\phi\to\psi),  \s\in\wh f}.
\ee
\end{definition}
The set of admissible saddle points is very important to compute the prefactor of the mean transition times. Near these points the process spends the most crucial time as it passes from a basin of attraction to another one.
\smallskip

We now present the assumptions on $S$.
\begin{assumption}\label{hyp:S}
We suppose that:
\begin{itemize}
\item $S$ has a finite number of minima and saddle points.
\item All the minima and saddle points of $S$ are non-degenerate (i.e. hyperbolic): at each point, the Hessian operator has non-zero eigenvalues.
\end{itemize}
\end{assumption}

Assumptions \ref{hyp:S} are structural. The finite number of stationary points provides a simple generalization of the case where there is only one saddle point. The non-degeneracy condition is necessary in order to approximate locally at the minima and saddle points the potential by its quadratic part. If this is not the case the prefactor in \eqref{eq:arrhenius} is not a constant but should have a dependence in $\epsilon$. 

Connections between Assumptions \ref{hyp:V} and \ref{hyp:S} are not straightforward. Proving that a given potential $S$ satisfies Assumption \ref{hyp:S} is not easy, a precise analysis is often needed. Moreover if we want to investigate the dependence of the potential $S$ on the parameter $\g$, bifurcations can occur and the landscape do not satisfy Assumption \ref{hyp:S} for some critical values of $\g$. See Berglund, Fernandez and Gentz \cite{berglund107, berglund207} for the finite and infinite dimensional cases for the double well potential. However, results exist (see \cite{angenent86} and references therein) on the generality of Assumption \ref{hyp:S}. 

In addition, under Assumptions \ref{hyp:S} and \ref{hyp:V}, the deterministic dynamical system (i.e. \eqref{eq:spde.1} without the white noise) satisfies a Morse-Smale structure (see \cite{brunovsky.fiedler89, fiedler.rocha96} and the references therein). This means that the attractor of the dynamical system consists of equilibria and heteroclinic orbits connecting these equilibria. Methods has been developed by Fiedler and Rocha in \cite{fiedler.rocha96}, by Wolfrum in \cite{wolfrum02} to compute the global attractor of the deterministic system.


\begin{remark}\label{rem:potential}
$H^1$ is the convenient functional space for the process since $S(\phi)<+\infty$ if and only if $\phi$ is in $H^1([0,1])$.
In fact from the upper bound \eqref{eq:V.2} and lower bound \eqref{eq:V.1} on $V$ we get
\begin{align}\label{eq:potential.3}
C_1(\norm{\phi}^2_{H^1}-1)&\ieg S(\phi)\ieg 
C'_1(\norm{\phi}^2_{H^1}+\norm{\phi}^p_{H^1}+1).
\end{align}
Each function in $H^1([0,1])$ is continuous and even $\al$-H\"older continuous (for $0<\al<\frac12$).
\end{remark}
\smallskip

For each $\phi\in C([0,1])$, we define the quantity $\Det(\cH_{\phi} S)$:
\begin{itemize}
\item for Dirichlet boundary conditions, let $f$ be the solution on $[0,1]$ of
\begin{align}\label{eq:fdet.1}
\cH_{\phi} Sf&=0&f(0)&=1&f'(0)=0
\end{align}
then $\Det(\cH_{\phi} S)=f(1)$
\item for Neumann boundary conditions, let $f$ be the solution on $[0,1]$ of
\begin{align}\label{eq:fdet.3}
\cH_{\phi} Sf&=0&f(0)&=0&f'(0)=1
\end{align}
then 
$\Det(\cH_{\phi} S)=f'(1)$.
\end{itemize}

Let us recall that, as a regular Sturm-Liouville operator, $\cH_{\phi}S$ has a countable number of eigenvalues, all of them real. We denote by $(\la_k(\phi))_{k\seg 1}$ the sequence of these eigenvalues in the increasing order. The definition of $\Det(\cH_{\phi} S)$ is justified by the following lemma. 
\begin{lemma}[\cite{levit.smilansky77}]\label{lem:fdet}
For any $\phi$ and $\psi$ with non-degenerate Hessian operator, the infinite product $\prod_{k=1}^{\infty}\frac{\la_k(\phi)}{\la_k(\psi)}$ is convergent and
we have
\be\label{eq:fdet.5}
\prod_{k=1}^{\infty}\frac{\la_k(\phi)}{\la_k(\psi)}=\frac{\Det(\cH_{\phi} S)}{\Det(\cH_{\psi} S)}.
\ee
\end{lemma}
This lemma relates the infinite product of the ratio of eigenvalues to a ratio of terminal values of solutions. We find an elementary proof in \cite{levit.smilansky77} by Levit and Smilansky  which relies on two different expressions of the Green function associated to the problem $\cH_{\phi}Sf=0$ satisfying the boundary conditions. In fact, the Green function could either be expressed using the spectral decomposition of $\cH_{\phi}S$ or expressed as a linear combination of two well-chosen fundamental solutions (of the second order linear differential equation).

\subsection{Main results}\label{ssec:main}

Before stating the main result, we describe the set of minima and saddle points. In fact, the prefactor depends greatly on the geometry of a graph connecting the minima to each other through the saddle points (so-called the $1$-skeleton connection graph by Fiedler and Rocha in \cite{fiedler.rocha10}). We define this graph and express the prefactor partly as an equivalent conductance on this graph.

We denote by $\cM$ the set of minima of $S$.
Since by Assumption \ref{hyp:S}, there is a finite number of stationary points, we order the minima by increasing energy. We denote by $\f_1, \f_2,\dots,\f_m$, $m=\abs{\cM}$, the different minima indexed by increasing energy
\be\label{eq:min.1}
S(\f_1)\ieg S(\f_2)\ieg\dots\ieg S(\f_m).
\ee
We denote by $\cM_l$, the subset of minima $\cM_l=\acc{\f_1, \f_2,\dots,\f_l}$ for $1\ieg l\ieg m$.

We consider the transitions from a minimum $\f_{l_0}$ to $\cM_l$ for $l<l_0$. 
These are the only visible metastable transitions. We will see from large deviations estimates, that to go from a minimum $\phi$ to another $\psi$, it requires a time of order $\exp\pare{\wh S(\phi,\psi)-S(\phi)/\e}$. The time required to make the reverse transition is also of order $\exp\pare{\wh S(\psi,\phi)-S(\psi)/\e}$.
Therefore if $S(\psi)>S(\phi)$, we get 
\be\label{eq:potential.4}
\wh S(\phi,\psi)-S(\phi)>\wh S(\psi,\phi)-S(\psi)
\ee
and the time required to go from $\phi$ to $\psi$ is much larger than for the reverse transition. So we cannot see the reverse transitions since there are absorbed by the direct ones.
If some minima have the same potential, we can suitably order them to consider a transition from one minimum to another one at a same height.

Let us now construct the weighted graph of paths from $\phi_{l_0}$ to $\cM_l$. We denote $\wh S=\wh S(\phi_{l_0},\cM_l)$ the common potential of the saddles. The minima $\cM$ are the vertices of the graph, the saddle points  in $\ccS(\phi_{l_0},\cM_l)$ are the edges. We connect an edge $\wh\s$ between two vertices  $\phi,\psi\in\cM$ if the saddle $\wh\s$ is a pass between the valleys of $\phi$ and $\psi$: there exists $f\in\G(\phi\to\psi)$ such that $\wh f$ has a unique element and $\wh f=\wh\s$. Existence of this graph is ensured by Assumptions \ref{hyp:S} (see \cite{fiedler.rocha10} and references therein).

Each saddle point in $\ccS(\phi_{l_0},\cM_l)$ has a unique negative eigenvalue from the Morse-Smale property and the hyperbolicity of the stationary points.
The weight associated to an edge $\wh\s$ is defined as
\be\label{eq:weigh.1}
w(\wh\s)=\frac{\abs{\la^-(\wh\s)}}{\sqrt{\abs{\Det \cH_{\wh\s}S}}}
\ee
where $\la^-(\wh\s)$ is the unique negative eigenvalue of $\cH_{\wh\s}S$.

$\wh\s^+$ and $\wh\s^-$ denote the two minima connected by a given edge $\wh\s$. Let us recall that we have $m$ minima in $\cM$. For a real valued vector $a\in \bbR^m$ indexed by the minima in $\cM$, we consider the following quadratic form
\be\label{eq:quadratic.1}
Q(a)=\sum_{\wh\s\in\ccS(\phi_{l_0},\cM_l)}w(\wh\s)(a(\wh\s^+)-a(\wh\s^-))^2.
\ee
We define $\cC^*(\phi_{l_0},\cM_l)$ the equivalent conductance of the graph between $\phi_{l_0}$ and $\cM_l$ as
\be\label{eq:quadratic.2}
\cC^*(\phi_{l_0},\cM_l)=\inf\acc{Q(a), a\in\bbR^m,a(\phi_{l_0})=1,a(\phi)=0, \phi\in\cM_l}.
\ee
This conductance is an approximation of the capacity between a neighborhood of $\phi_{l_0}$ and $\cM_l$. In some sense, we replace the continuous landscape defined by $S$ by a graph containing the relevant geometric structure of the landscape.

Let us denote by $\cB_{\rho}(\phi)$, for $\phi\in H^1_{bc}[0,1]$, the ball of center $\phi$ and radius $\rho$ in $H^1_{bc}$
\be\label{eq:ball.1}
\cB_{\rho}(\phi)=\acc{\s\in H^1_{bc}, \norm{\s-\phi}_{L^2}\ieg\rho, \norm{\s}_{H^1}<A_1}
\ee
where $A_1$ is a sufficiently large constant.
We also define $\cB_{\rho}(\cM_l)=\cup_{\phi\in\cM_l}\cB_{\rho}(\phi)$.
We choose this kind of neighborhood because in the following we need to control the norm in the uniform norm and in the $\al$-H\"older norm (for $\al<\frac12$).
\bigskip

We now state our main result describing the dependence in $\e$ of the mean of the hitting time of a union of balls around the points of $\cM_l$ starting from $\phi_{l_0}$.

\begin{theorem}\label{th:main}
Under the assumptions \ref{hyp:V}, \ref{hyp:S}, for any minimum $\phi_{l_0}$, and a set of minima $\cM_l$ with $l_0>l$, there exists $\rho_0$ such that for any $\rho_0>\rho>0$
\be\label{eq:time.1}
\bbE_{\phi_{l_0}}[\tau_\e(\cB_{\rho}(\cM_l))]
=\frac{2\pi e^{{\wh S(\phi_{l_0},\cM_l)}/{\e}}}{{\cC^*(\phi_{l_0},\cM_l)}\sqrt{\Det \cH_{\phi_{l_0}}S}}(1+\Psi(\e))
\ee
where the error term satisfies $\Psi(\e)=O(\sqrt{\e}\abs{\ln(\e)}^{3/2})$. 
\end{theorem}

For the simple case where we have only three stationary points, two minima and one saddle (e.g. the case of the double well potential \eqref{eq:allen.cahn.1} with Neumann boundary conditions, for $\g>1/\pi^2$), we have the following corollary.
\begin{corollary}\label{cor:main}
Let $\phi^+$ and $\phi^-$ be the two minima with $S(\phi^-)\seg S(\phi^+)$ and $\wh\s$ the unique saddle point. There exists $\rho_0$ such that for any $\rho_0>\rho>0$
\be\label{eq:time.2}
\bbE_{\phi^-}[\tau_\e(\cB_{\rho}(\phi^+))]
=\frac{2\pi}{\abs{\la^-(\wh\s)}} \sqrt{\frac{\abs{\Det \cH_{\wh\s}S}}{\Det \cH_{\phi^-}S}} e^{(S(\wh\s)-S(\phi^-))/{\e}}(1+\Psi(\e))
\ee
where the error term is $\Psi(\e)=O(\sqrt{\e}\abs{\ln(\e)}^{3/2})$. 
\end{corollary}

\subsection{Sketch of proof of Theorem \ref{th:main}}\label{ssec:sketch}

We first introduce the discretization we consider. 
The finite dimensional approximation of the SPDE is constructed as in the work of Funaki \cite{funaki83} and the work of Gy\"ongy\cite{gyongy98}. The approximation is defined via a spatial finite difference approximation of Equation \eqref{eq:spde.1}.

We denote by $S_N$ the discretized potential, for $y\in\bbR^{N+2}$
\be\label{eq:potentiel.3b}
S_N(y)=h_N\sum_{i=0}^N\frac{\g}{2h_N^2}(y_{i+1}-y_i)^2+V(y_i)
\ee
where $h_N>0$ is the step of discretization. We set $X^i_0=u_0(x_i)$ where $u_0\in C_{bc}([0,1])$ is the initial condition and the $x_i$ are the discretization points on $[0,1]$. Let us denote by $x_{i-1/2}$ the middle point of $[x_{i-1},x_i]$.
We construct a $N$-dimensional Brownian motion $B$ from the white noise $W$. Doing so we will be able to prove the convergence of $u^N$ to $u$ in $L^p$ and almost surely.
Thus we define, for $1\ieg i\ieg N$ 
\be\label{eq:bm.1}
B_t^i
=\frac1{\sqrt{h_N}}W\pare{\cro{x_{i-1/2},x_{i+1/2}}\times[0,t]}.
\ee
The properties of the white noise imply that $(B^i)$ are independent Brownian motions.

The $N$-dimensional process $(X_t)_t$ is the solution of 
\be\label{eq:finite.2b}
\dint X^i_t=- \frac1{h_N}\grad S_N(X_t)^i\dt +\sqrt{\frac{2\e}{h_N}}\dint B^i_t \text{ for $i=1..N$}.
\ee
$X^0$ and  $X^{N+1}$ are defined by the boundary conditions
\begin{itemize}
\item for Dirichlet boundary conditions:
\be\label{eq:bc.dN}
X^0_t=X^{N+1}_t=0,\quad\forall t\seg0
\ee
\item for Neumann boundary conditions:
\begin{align}\label{eq:bc.nN}
X^0_t&=X^1_t&\text{ and }&&X^{N+1}_t&=X^{N}_t,\quad\forall t\seg0.
\end{align}
\end{itemize}
\smallskip

The discretized system $u^N$ is the linear interpolation between the points $(x_i,X^i)$.  To simplify, it is easier to adapt the parameters to the boundary conditions.
\begin{itemize}
\item For Dirichlet boundary conditions, we choose
\begin{align}\label{eq:points.d}
h_N&=\frac1{N+1},&x_i&=\frac i{N+1},\quad\forall 0\ieg i\ieg N+1.
\end{align}
\item For Neumann boundary conditions, we choose
\begin{align}\label{eq:points.n}
h_N&=\frac1{N},&x_i&=\frac i{N}-\frac1{2N},\quad\forall 0\ieg i\ieg N+1.
\end{align}
\end{itemize}
\smallskip

We set $\tau^N_\e(B)$ the hitting time of a set $B$ for the discretized system
\be\label{eq:hitting.N.1}
\tau^N_\e(B)=\inf\acc{t>0,u^N(N^{-1}t)\in B}.
\ee
We decompose the proof of Theorem \ref{th:main} in several steps: 
\begin{enumerate}
\item for a given $\e$ and a sequence of initial conditions $\phi_{l_0}^N$, each being a minimum of $S^N$, converging to $\phi_{l_0}$ (see Proposition \ref{prop:stationary}), we prove that the expectation of $\tau^N_\e(\cB_{\rho}(\cM_l))$ converges to the expectation of the hitting time for the SPDE:
\be\label{eq:conv.time.1}
\lim_{N\to\infty}\bbE_{\phi_{l_0}^N}[\tau^N_\e(\cB_{\rho}(\cM_l))]
=\bbE_{\phi_{l_0}}[\tau_\e(\cB_{\rho}(\cM_l))].
\ee
To this aim, we use the convergence of $u^N$ to the solution $u$.
This is done in Section \ref{sec:conv}.

\item For a fixed $N$, we compute the asymptotics of the transition time uniformly on the dimension. We get a prefactor $a_N(\e)$ such that
\be\label{eq:time.N.1}
\abs{\frac1{a_N(\e)}\bbE_{\phi_{l_0}^N}[\tau^N_\e(\cB_{\rho}(\cM_l))]-1}
=\psi(\e,N)<\Ps(\e)=O(\sqrt{\e}\abs{\ln(\e)}^{3/2})
\ee
where the error term $\Psi(\e)$ does not depend on $N$. This step is the main estimate and is detailed below.

\item The limit $N\to\infty$ of $a_N(\e)$ gives us the correct asymptotics for the transition time in the infinite dimensional case:
\be\label{eq:prefactor}
a(\e)=\lim_{N\to\infty}a_N(\e).
\ee
This is done in Section \ref{sec:prefactor}.
\end{enumerate}

The estimate \eqref{eq:time.N.1} is proved in two steps. 
\begin{enumerate}
\item[(i)] First we start from a probability measure (the equilibrium probability: $\n^N$) on the boundary of a chosen neighborhood of the minimum $\phi^N_{l_0}$, which allows us to do the computation of $a_N(\e)$:
\be\label{eq:time.N.2}
\abs{\frac1{a_N(\e)}\bbE_{\n^N}[\tau_\e(\cB_{\rho}(\cM_0))]-1}
=\psi_1(\e,N)<\Ps_1(\e)=O(\sqrt{\e}\abs{\ln(\e)}^{3/2}).
\ee
This is done in Section \ref{sec:estimates}.

\item[(ii)] Then we have to control the error made by starting on the boundary of the minimum and not precisely at the minimum:
\begin{align}\label{eq:time.N.3}
\frac1{a_N(\e)}\abs{\bbE_{\n^N}[\tau_\e(\cB_{\rho}(\cM_0))]
-\bbE_{\phi_{l_0}^N}[\tau^N_\e(\cB_{\rho}(\cM_0))]}
=\psi_2(\e,N)<\Ps_2(\e)
\end{align}
with $\Ps_2(\e)=O(\sqrt{\e}\abs{\ln(\e)}^{3/2})$.
This result comes from the loss of memory of the initial condition adapted from Martinelli in \cite{martinelli89}. This is exposed in Section \ref{sec:initial}.
\end{enumerate}

\section{Discretization}\label{sec:conv}

In this section, we present the convergence of the discretization $u^N$ to the solution of the SPDE and prove the convergence of the hitting times.

\subsection{Finite Dimensional Model}\label{ssec:sde}

We write the discretized system $u^N$ in a mild form.
We define a function $\ka_N$, with $\floor{x}$ the integer part of $x$,
\begin{align}\label{eq:discrete.2d}
\ka_N(x)&=\frac{\floor{(N+1)x+\frac12}}{N+1},&\text{for Dirichlet boundary conditions,}\\
\label{eq:discrete.2n}
\ka_N(x)&=\frac{\floor{Nx}+1}{N}-\frac1{2N},&\text{for Neumann boundary conditions.}
\end{align}

We define $g^N$ the semi-group associated with the discretized Laplacian.
The discretized Laplacian is a $N$ dimensional matrix, denoted by $\D_d^N$ for Dirichlet boundary conditions and by $\D_n^N$ for Neumann boundary conditions: 
\begin{align}\label{eq:laplacian.1}
\D_d^N&=\frac1{h_N^2}
\begin{pmatrix}-2&1&0&\dots&0\\
1&-2&\ddots&\ddots&\vdots\\
0&\ddots&\ddots&\ddots&0\\
\vdots&\ddots&\ddots&\ddots&1\\
0&\dots&0&1&-2
\end{pmatrix}
&
\D_n^N&=\frac1{h_N^2}
\begin{pmatrix}-1&1&0&\dots&0\\
1&-2&\ddots&\ddots&\vdots\\
0&\ddots&\ddots&\ddots&0\\
\vdots&\ddots&\ddots&\ddots&1\\
0&\dots&0&1&-1
\end{pmatrix}.
\end{align}
We consider the matrix
$p^N(t)=h_N^{-1}e^{t\g\D^N}$. 
Therefore
$p^N(t)_{i,j}$ is the solution of
\be\label{eq:exp.2}
\begin{cases}
\dfrac{\dint}{\dt} p^N(t)_{i,j}&=(\g\D^Np^N(t))_{i,j}\\
p^N(0)_{i,j}&=\dfrac1{h_N}\de_{ij}.
\end{cases}
\ee
The semi-group $g^N$  is the linear interpolation of $p^N(t)$ on $[0,1]\times[0,1]$ along the discretization points.

Let us now prove the convergence of the solution of \eqref{eq:mild.4} to the solution of Equation \eqref{eq:spde.1}.

\begin{theorem}\label{th:conv.spde}
For all initial condition $u_0\in C_{bc}^3([0,1])$, $T>0$, and $p\seg 1$, we get the convergence
\be\label{eq:conv.1}
u^N\conv{N}{\infty}u\quad\text{on $[0,1]\times[0,T]$}
\ee
in the following senses:
\begin{itemize}
\item in $L^p(\Om,C([0,1]\times[0,T]))$, i.e. $\bbE\cro{\norm{u^N-u}_{\infty,T}^p}^{\frac1p}\conv{N}{\infty}0$
\item almost surely in $C([0,1]\times[0,T])$, i.e. for every $\eta\in]0,\frac12[$, there exists $\Xi$ a random variable almost surely finite such that
\be\label{eq:conv.2}
\norm{u^N-u}_{\infty,T}\ieg\frac{\Xi}{N^{\eta}}.
\ee
\end{itemize}
\end{theorem}

\begin{remark}\label{rem:conv}
Let us denote
\be\label{eq:norm.2}
\norm{u}_{q,T}=\sup_{t\in[0,T]}\cro{\int_0^1\abs{u(x,t)}^q\dx}^{\frac1q}
=\sup_{t\in[0,T]}\norm{u(t)}_{L^q}.
\ee
We have $\norm{u}_{q,T}\ieg \norm{u}_{\infty,T}$.
As a consequence we get convergence in Theorem \ref{th:conv.spde} in the $L^q$ norm instead of the uniform norm.
\end{remark}

The convergence of the finite discretization is proved in \cite{gyongy98} if $V'$ is globally Lipschitz. We proved that the result holds in the case that $V'$ satisfies \eqref{eq:drift.1} via a localization argument. 
The idea, notably used by Funaki in \cite{funaki83}, is to rewrite the finite dimensional system $u^N$ in a "mild form" and prove the convergence of this finite dimensional mild form to the infinite dimensional mild form \eqref{eq:mild.3}.

\begin{lemma}\label{lem:mildN}
For every $u_0\in C_{bc}([0,1])$ and $N>0$, the function $u^N$ defined on  $[0,1]\times\bbR^+$ satisfies the equation
\begin{multline}\label{eq:mild.4}
u^N(x,t)
=\int_0^1g^N_{t}(x,\ka_N(y))u_0(\ka_N(y))\dy
-\int_0^t\int_0^1g^N_{t-s}(x,\ka_N(y))V'(u^N(\ka_N(y),s))\dy\ds\\
+\sqrt{2\e}\int_0^t\int_0^1g^N_{t-s}(x,\ka_N(y))W(\dy,\ds).
\end{multline}
For all $p\seg 1$ and $T>0$, we have
\be\label{eq:moment.2}
\sup_N\bbE\Big[\sup_{[0,T]\times[0,1]}\abs{u^N(x,t)}^p\Big]\ieg C(T,p).
\ee
\end{lemma}

\begin{proof}
This lemma is just a reformulation of the system of stochastic differential equations. We use the variation of the constant to integrate the linear part and then interpolate linearly the system to obtain a mild formulation of the function $u^N$ (see \cite{funaki83,gyongy98}). To obtain the uniform moment bound, we proceed classically using a truncation procedure. We define $u^N_R$ and $u_R$ solutions of equations \eqref{eq:mild.4} and \eqref{eq:mild.3} in which we have replaced the function $V'$ by $b_R$ defined, for $R>0$ by
\be\label{eq:local.1}
b_R(u)=V'(u)\1_{[-R,R]}+V'(R)\1_{]R,+\infty[}+V'(-R)\1_{]-\infty,-R[}.
\ee
$b_R$ is continuous, bounded and globally Lipschitz.
Firstly, using the uniform estimates of the semi-group and the boundedness of $b_R$, we prove that for all $T$, all $p>1$, there exists $C(p,T,R)$ independent of $N$ such that
\be\label{eq:moment.3}
\sup_{[0,1]\times[0,T]}\bbE\cro{\abs{u_R^N(x,t)}^p}\ieg C(p,T,R)<+\infty.
\ee
Secondly, there exists $C(p,T,R)$ independent of $N$, such that
\be\label{eq:moment.4}
\sup_N\bbE\Big[\sup_{[0,1]\times[0,T]}\abs{u_R^N(x,t)}^p\Big]
\ieg C(p,T,R)<+\infty.
\ee
We use regularity of the solution (Kolmogorov's theorem) to prove \eqref{eq:moment.4}.
Thirdly, we use a comparison theorem to obtain uniform bounds on $u^N$ from bounds on $u^N_{R_0}$ where $R_0$ is fixed and sufficiently large.
\end{proof}

We use the convergence of $u^N_R$ to $u_R$ proved in \cite{gyongy98}.
\begin{proposition}\label{prop:conv.local}[\cite{gyongy98}]
For all $R>0$, $T>0$ and $0<\eta<\frac12$ and $u_0$ in $C^3_{bc}[0,1]$, there exists a random variable $\xi_R$ almost surely finite such that
\be\label{eq:conv.5}
\norm{u_R^N-u_R}_{\infty,T}
\ieg\frac{\xi_R}{N^{\eta}}.
\ee
\end{proposition}

\begin{proof}[Proof of Theorem \ref{th:conv.spde}]
Let $R>0$, we define the stopping times
\begin{align}\label{eq:time.local}
\tau_R&=\inf\{t, \norm{u_R(t)}_{\infty}>R\}=\inf\{t, \exists x\in[0,1], \abs{u_R(x,t)}>R\}\\
\tau^N_R&=\inf\{t, \norm{u^N_R(t)}_{\infty}>R\}=\inf\{t, \exists x\in[0,1], \abs{u^N_R(x,t)}>R\}.
\end{align}

Let us choose $0<\de<1$. For $R>1$, we define
\be\label{eq:conv.13}
\Om_R=\{\tau_{R-\de}> T\text{ and }\liminf_{N\to\infty}\tau_R^N> T\}.
\ee
First we show that $\bbP[\Om_R]\conv{R}{\infty}1$.
Let $M>0$. For $\om\in\{\xi_R<M\}\cap\{\tau_{R-\de}\seg T\}$, by Proposition \ref{prop:conv.local}, for $N$ sufficiently large,
\be\label{eq:conv.8}
\norm{u^N_R}_{\infty,T}(\om)<\norm{u_R}_{\infty,T}(\om)+\de<R
\ee
which means that $\liminf_{N\to\infty}\tau^N_R(\om)\seg T$.
Then by taking the complement relatively to $\{\xi_R<M\}$ we get
\begin{align}\label{eq:conv.10}
\bbP[\liminf_{N\to\infty}\tau^N_R< T;\xi_R<M]
&\ieg\bbP[\tau_{R-\de}< T;\xi_R<M]
\ieg\bbP[\tau_{R-\de}< T].
\end{align}
By definition of the time $\tau_{R-\de}$, we have by the Markov inequality for $p>1$ and from Equation \eqref{eq:moment.1}
\begin{align}\label{eq:conv.11}
\bbP[\liminf_{N\to\infty}\tau^N_R< T;\xi_R<M]
&\ieg\bbP[\tau_{R-\de}\ieg T]
\ieg\bbP\cro{\norm{u}_{\infty,T}\seg R-\de}
\ieg\frac{\bbE\cro{\norm{u}_{\infty,T}^p}}{(R-\de)^p}.
\end{align}
Finally we get
\begin{align}\label{eq:conv.12}\nonumber
\bbP[\Om^c_R]
&=\bbP[\tau_{R-\de}\ieg T\text{ or }\liminf_{N\to\infty}\tau_R^N\ieg T]\\\nonumber
&\ieg\bbP[\tau_{R-\de}\ieg T]+\bbP[\liminf_{N\to\infty}\tau^N_R< T;\xi_R<M] +\bbP[\xi_R\seg M]\\
&\ieg\frac{2\bbE[\norm{u}_{\infty,T}^p]}{(R-\de)^p}+\bbP[\xi_R\seg M].
\end{align}
Since $\xi_R$ is finite almost surely, we take first the limit $M\to+\infty$ then $R\to+\infty$.
\smallskip

Let us define $\wt{\Om}_R=\Om_R\cap\{\xi_R<\infty\}$. Since $\tau_R$ and $\tau_R^N$ are increasing in $R\in\bbN$, the sets $\Om_R$ are also increasing in $R$. Then we have
\be\label{eq:conv.13b}
\bbP[\cup_{R>1}^{\infty}\wt{\Om}_R]=\bbP[\cup_{R\in\bbN}\Om_R]=\lim_{R\to\infty}\bbP[\Om_R]=1.
\ee
Let $\om\in\wt{\Om}_R$. By definition of $\tau^N_R$, there exists $N_0(\om)$ such that for all $N\seg N_0(\om)$, $\tau^N_R(\om)>T$ and $\tau_{R-\de}(\om)>T$.
By using the proposition \ref{prop:conv.local}, for all $N\seg N_0(\om)$,
\begin{align}\label{eq:conv.15}
\norm{u^N-u}_{\infty,T}(\om)=\norm{u_R^N-u_R}_{\infty,T}(\om)\ieg\xi_R(\om)N^{-\eta}.
\end{align}
We define $\xi'_R(\om)$  by
\be\label{eq:conv.16}
\xi'_R(\om)=\sup_{N\ieg N_0(\om)}N^{\eta}\norm{u_R^N-u_R}_{\infty,T}(\om)+\xi_R(\om).
\ee
$\xi'_R(\om)$ is finite on $\wt{\Om}_R$ and is such that $\norm{u^N-u}_{\infty,T}\ieg\xi'_RN^{-\eta}$.
Let us define the random variable $\Xi$ by
\begin{align}\label{eq:conv.17}\nonumber
\Xi(\om)&=\xi'_R(\om) \text{ on }\wt{\Om}_R\setminus\wt{\Om}_{R-1}\text{ for }R\seg 2\\
\Xi(\om)&=\xi'_1(\om) \text{ on }\wt{\Om}_1.
\end{align}
Then on  $\cup_{R\seg 1} \wt{\Om}_R$, set of probability $1$, $\Xi$ is almost surely finite and $\norm{u^N-u}_{\infty,T}\ieg\Xi N^{-\eta}$ which finishes the proof of the almost sure convergence. 

To conclude, we show that $\bbE\cro{\norm{u^N-u}_{\infty,T}^p}$ converges to $0$.
Since $\norm{u^N}_{\infty,T}$ has uniform moments in $N$ (Lemma \ref{lem:mildN}), we define
\begin{align}\label{eq:conv.20}
\Om_{R,N_0}&=\cap_{N\seg N_0}\{\tau_{R-\de}> T\text{ and }\tau_R^N> T\}.
\end{align}
We have $\Om_R=\cup_{N_0}\Om_{R,N_0}$.
For all $N\seg N_0$, we get by definition
\begin{align}\label{eq:conv.21}
\norm{u^N-u}_{\infty,T}^p
&=\1_{\Om_{R,N_0}}\norm{u_R^N-u_R}_{\infty,T}^p+\1_{\Om^c_{R,N_0}}\norm{u^N-u}_{\infty,T}^p.
\end{align}
Thus using Cauchy-Schwarz inequality and the bound \eqref{eq:moment.2}, we get
\begin{align}\label{eq:conv.22}
\bbE\cro{\norm{u^N-u}_{\infty,T}^p}
&\ieg\bbE\cro{\norm{u_R^N-u_R}_{\infty,T}^p}+\bbP[\Om^c_{R,N_0}]^{\frac12}C(2p,T)^{\frac12}.
\end{align}
Using the convergence of $u^N_R$ to $u^R$ (Proposition \ref{prop:conv.local}), we obtain
\be\label{eq:conv.23}
\limsup_{N\to\infty}\bbE\cro{\norm{u^N-u}_{\infty,T}^p}\ieg C(2p,T)^{1/2}\bbP[\Om^c_{R,N_0}]^{\frac12}.
\ee
Let us fix $\eta>0$. Since $\bbP[\Om_R]$ tends to $1$ and $\Om_R$ is increasing, we choose $R$ such that $\bbP[\Om_{R}^c]\ieg\eta$.
Similarly, $\Om_{R,N_0}$ is increasing in $N_0$, thus 
$\bbP[\Om_{R}^c]
=\lim_{N_0\to\infty}\bbP[\Om_{R,N_0}^c]
\ieg\eta$.
Let us choose $N_0$ such that $\bbP[\Om_{R,N_0}^c]\ieg2\eta$.
Inserting this bound in \eqref{eq:conv.23}, we obtain the result.
\end{proof}

\subsection{Convergence of the Transition Times}\label{ssec:conv.time}

We conclude this section by proving the convergence of the transition times.

Let us denote by $u_0$ the initial condition of the solution of Equation \eqref{eq:spde.1} and $\phi$ a continuous function. We define the hitting times: for $\rho>0$
\begin{align}\label{eq:time.1b}
\tau_{\e}(\rho)&=\inf\acc{t>0,\norm{u(t)-\phi}_{\infty}<\rho}\\
\tau_{\e}^N(\rho)&=\inf\acc{t>0,\norm{u^N(t)-\phi^N}_{\infty}<\rho}
\end{align}
where $\phi^N$ is the linear approximation of $\phi$.

\begin{proposition}\label{prop:conv.time}
Suppose that $\norm{\phi^N-\phi}_{\infty}$ converges to $0$ and that there exists $\rho_0$ such that for every $\rho_0>\rho>0$, 
\be\label{eq:time.0}
\bbE_{u_0}[\tau_{\e}(\rho)]<\infty.
\ee
Then for almost every $\rho>0$, 
\be
\tau_{\e}^N(\rho)\conv{N}{\infty}\tau_{\e}(\rho)\quad\text{a.s.}
\text{ and}\quad
\bbE_{u^N_0}[\tau_{\e}^N(\rho)]\conv{N}{\infty}\bbE_{u_0}[\tau_{\e}(\rho)].
\ee
\end{proposition}

\begin{proof}For the sake of simplicity we omit $\e$ in the proof.
First we prove that for all $\de>0$, $T>0$, we have
\be\label{eq:time.2b}
\tau(\rho+\de)\wedge T
\ieg \liminf_{N\to\infty}\tau^N(\rho)\wedge T
\ieg \limsup_{N\to\infty}\tau^N(\rho)\wedge T
\ieg\tau(\rho-\de)\wedge T
\text{ a.s.}
\ee

From Theorem \ref{th:conv.spde},
$\norm{u^N-u}_{\infty,T}$ converges to $0$  almost surely.
Therefore with probability $1$, there exists $N_0(\om)$ such that for all $N\seg N_0(\om)$
\begin{align}\label{eq:time.4}
\sup_{t\in[0,T]}\norm{u^N(t)-u(t)}_{\infty}(\om)
&<\frac{\de}2&\text{and}&&
\norm{\phi^N-\phi}_{\infty}&<\frac{\de}2.
\end{align}
Then for $t\ieg \tau(\rho+\de)\wedge T$ and $N\seg N_0(\om)$, using the triangle inequality we get
\begin{align}\label{eq:time.5}\nonumber
\rho+\de
\ieg\norm{u(t)-\phi}_{\infty}
&\ieg\norm{u(t)-u^N(t)}_{\infty}+\norm{u^N(t)-\phi^N}_{\infty}+\norm{u^N_f-\phi}_{\infty}\\
&\ieg \de+\norm{u^N(t)-\phi^N}_{\infty}
\end{align}
which means that $t\ieg \tau^N(\rho)\wedge T$. Thus, we obtain
$\tau(\rho+\de)\wedge T\ieg\liminf_{N\to\infty}[\tau^N(\rho)\wedge T]$ almost surely.
By the same arguments for $t\ieg \tau^N(\rho)\wedge T$ and $N\seg N_0(\om)$, we get
\begin{align}\label{eq:time.7}
\rho
\ieg\norm{u^N(t)-\phi^N}_{\infty}
&\ieg \de+\norm{u(t)-\phi}_{\infty}.
\end{align}
Therefore $\limsup_{N\to\infty}[\tau^N(\rho)\wedge T]\ieg\tau(\rho-\de)\wedge T$
which proves the inequality \eqref{eq:time.1b}.

From the definitions of $\tau(\rho)$ and $\tau^N(\rho)$, the functions $\rho\mapsto\tau(\rho)$ and $\rho\mapsto\tau^N(\rho)$ are left continuous and have right limits.
Then using the fact that $\tau(\rho)$ is finite almost surely, we get
\be\label{eq:time.9}
\tau(\rho^+)
\ieg \liminf_{N\to\infty}\tau^N(\rho)
\ieg \limsup_{N\to\infty}\tau^N(\rho)
\ieg\tau(\rho)<+\infty\text{ a.s.}
\ee
where $\tau(\rho^+)=\lim_{\de\to 0^+}\tau(\rho+\de)$.

At a point of continuity of $\rho\mapsto\tau(\rho)$, we obtain $\tau(\rho)= \lim_{N\to\infty}\tau^N(\rho)$.
Let us fix $\rho_1>0$. There exists $\ccN\subset\Om$ a null set such that for $\om\notin \ccN$,  $\rho\mapsto\tau(\rho)(\om)$ is bounded, decreasing, left continuous on $[\rho_1,+\infty[$. We define the set of discontinuities, $\ccP$:
\begin{align}\label{eq:time.11}
\ccP&=\acc{(\om,\rho)\in\ccN^c\times[\rho_1,+\infty[, \tau(\rho^+)(\om)\neq\tau(\rho)(\om)}\subset\Om\times\bbR.
\end{align}

Then we consider the projection $\Pi^{\bbR}_{\om}$ from $\Om\times\bbR$ on $\bbR$ along $\acc{\om}\times\bbR$. For $\om\in\ccN^c$ we define
\be\label{eq:time.11b}
\cD(\om)=\Pi^{\bbR}_{\om}(\ccP)
=\acc{\rho\in[\rho_1,+\infty[, \tau(\rho^+)(\om)\neq\tau(\rho)(\om)}
\subset \bbR.
\ee
$\cD(\om)$ is at most countable since $\rho\mapsto\tau(\rho)(\om)$ is a bounded decreasing function. 

We define $\cN(\rho)=\Pi^{\Om}_{\rho}(\ccP)$ with $\Pi^{\Om}_{\rho}$ the projection from $\Om\times\bbR$ on $\Om$ along $\Om\times\{\rho\}$. $\cN(\rho)$ is the set of $\Om$ for which $\tau(\rho)$ is not continuous at $\rho$. Therefore, we have
\be\label{eq:time.12b}
\ccP=\cup_{\om\in \Om}\acc{\om}\times\cD(\om)
=\cup_{\rho>\rho_1}\cN(\rho)\times\acc{\rho}.
\ee
Then, using Fubini-Tonelli Theorem
\be\label{eq:time.13}
\int_{\rho_1}^{+\infty}\bbP[\cN(\rho)]\dint{\rho}
=\int_{\Om}\int_{\rho_1}^{+\infty}\1_{\ccP}(\om,\rho)\dint{\rho}\dint \bbP(\om)
=\int_{\Om}\int_{\rho_1}^{+\infty}\1_{\cD(\om)}(\rho)\dint{\rho}\dint \bbP(\om)
=0.
\ee
We get a null set $\cE(\rho_1)$ on $[\rho_1,+\infty[$ such that
$\bbP[\cN(\rho)]=0$ for all $\rho\in \cE(\rho_1)$ i.e. the convergence is almost sure.
To conclude, we consider a sequence $(\rho_n)_{n\seg 0}$ converging to $0$, then $\cE=\cup_{n\seg 0}\cE(\rho_n)$ is a null set of $\bbR$ on which the convergence is almost sure.

By using dominated convergence, we obtain the convergence of the expectations.
\end{proof}

\section{Initial condition}\label{sec:initial}


\subsection{Large Deviation Control}\label{ssec:ld}

For $0<\al<1$, we set $C^{\al}([0,1])$ the set of $\al$-H\"older continuous functions on $[0,1]$ equipped with the norm $\norm{\cdot}_{C^{\al}}$
\begin{align}\label{eq:holder.1}
\norm{f}_{C^{\al}}&=\norm{f}_{\infty}+\sup_{x, y}\frac{\abs{f(x)-f(y)}}{\abs{x-y}^{\al}}.
\end{align}
We also define $D^{\al}([0,1])$ the separable subset of this H\"older space which is the closure of $C^{\infty}$ in $C^{\al}$.


Let $0< \al<\frac12$ and $\rho>0$, we consider the neighborhood $B_{\rho}^{\al}(\phi)$ of $\phi\in D^{\al}_{bc}([0,1])$
\be
B_{\rho}^{\al}(\phi)
=\acc{\psi\in D_{bc}^{\al}([0,1]), \norm{\phi-\psi}_{C^{\al}}<\rho}.
\ee
We also have $B_{\rho}^{\al}(\cM_l)=\cup_{\phi\in\cM_l}B_{\rho}^{\al}(\phi)$.

With this large deviation principle, Chenal and Millet \cite{chenal.millet97} derive exponential asymptotic estimates for the exit time of domains with a unique stable stationary point. Using their evaluations and the procedure developed by Freidlin-Wentzell \cite{freidlin.wentzell84} in the finite dimensional case, we have the following result.
\begin{lemma}[\cite{chenal.millet97}]\label{lem:ldp}
For $0<\al<\frac12$, there exists $\rho_0$ such that for all $\rho<\rho_0$, we have for all $\phi\in B_{\rho}^{\al}(\phi_{l_0})$ and $\eta>0$
\begin{align}\label{eq:ldp.1}
\lim_{\e\to0}\bbP_{\phi}\cro{\exp\pare{\e^{-1}(\wh S+\eta)}>\tau_{\e}(B_{\rho}^{\al}(\cM_l))>\exp\pare{\e^{-1}(\wh S-\eta)}}&=1,
\end{align}
where $\wh S=\wh S(\phi_{l_0},\cM_l)$.
Let $\tau_{\e}=\tau_{\e}(B_{\rho}^{\al}(\cM_l))$. Then
\be\label{eq:ldp.2}
\frac{\tau_{\e}}{\bbE_{\phi}\cro{\tau_{\e}}}
\lawconv{\e}{0}
\cE
\ee
where $\cE$ is an exponential variable of parameter $1$.
Moreover for all $\phi\in B_{\rho}^{\al}(\phi_{l_0})$
\begin{align}\label{eq:ldp.3a}
\lim_{\e\to0}\e\log\bbE_{\phi}\cro{\tau_{\e}}&=\wh S&\text{ and }&&
\lim_{\e\to0}\e\log\bbE_{\phi}\cro{\tau_{\e}^2}&=2\wh S.
\end{align}
\end{lemma}
These estimates are the infinite dimensional version of the Freidlin-Went\-zell theory. 

\subsection{Exponential Contractivity}\label{ssec:contractivity}

For a given $\psi^N=(\psi_1,\cdots,\psi_N)\in \bbR^N$, we consider equivalently the point in $\bbR^N$ and the function in $C([0,1])$ obtained by the linear interpolation between the points $(x_i,\psi_i)$.
Reciprocally, for $\psi\in C_{bc}([0,1])$, we let $\wh\psi^N$ be the linear interpolation of $\psi$ along the discretization. $\wh\psi^N$ is the linear interpolation between the points $(x_i,\psi(x_i))$.

We set
\be
B^{\infty}_{\rho}(\phi)
=\acc{\psi\in C_{bc}([0,1]),\norm{\psi-\phi}_{\infty}<\rho}.
\ee

We adapt trajectorial results of contractivity for the localized process from Martinelli and Scoppola \cite{martinelli89}. We denote $u(\phi), u_R(\phi)$ the solutions of Equation \eqref{eq:spde.1} with respectively $V'$ and $b_R$, starting from $\phi$. Accordingly, we denote $u^N(\phi^N), u^N_R(\phi^N)$ the solutions of Equation \eqref{eq:mild.4} with $V'$ and $b_R$, starting from $\phi^N\in \bbR^N$.

\begin{lemma}\label{lem:contr}
Let $\phi$ be a minimum of $S$ and $R\seg R_0$.
There exists $m,C_R>0$ and $\e_0,\rho_0>0$, such that for all $\rho<\rho_0$ and every $\psi\in B^{\infty}_{\rho}(\phi)$  we have, for all $\e_0>\e>0$
\be\label{eq:contr.1}
\bbP\cro{\sup_{N\seg N_0}\norm{u_R^N(\wh{\psi}^N)(t)-u_R^N(\wh\phi^N)(t)}_{\infty}\ieg e^{-mt}\norm{\psi-\phi}_{\infty},\forall t>0}
\seg 1-e^{-\frac{C_R}{\e}}.
\ee
\end{lemma}
This result can be proved via an adaptation of the arguments of \cite{sbano94} and \cite{martinelli89}. 
 Lemma \ref{lem:contr} describes that the solutions of Equation \eqref{eq:spde.1} and \eqref{eq:mild.4} depend slightly on the initial condition. Moreover, the solutions starting from two functions are exponentially close uniformly in the dimension. Martinelli and Scoppola called that the loss of memory of the initial condition because the specific initial condition is not relevant for the evolution of the process.

\subsection{Uniformity in the initial condition}\label{ssec:loss}

Let us recall that $\phi_{l_0}$ is a minimum and $\cM_l$ is a set of lower minima. We denote
\begin{align}\nonumber
\tau^N_{\e}(\phi_{l_0})
&=\tau_{\e}^N(B_{\rho}^{\al}(\phi_{l_0}))
=\inf\acc{t,u^N(t)\in B_{\rho}^{\al}(\phi_{l_0})}\\
\tau^N_{\e}(\cM_l)
&=\tau_{\e}^N(B_{\rho}^{\al}(\cM_l))
=\inf\acc{t,u^N(t)\in B_{\rho}^{\al}(\cM_l)}.
\end{align}
Similarly, we denote by $\tau^{N,R}_{\e}$ the hitting time associated with the localized process $u^N_R$.
\begin{proposition}\label{prop:initial}
For all $\rho_0>\rho>0$, there exists $\eta>0$ such that for a sequence $\phi_{l_0}^N$ of minima of $S^N$, converging to $\phi_{l_0}$ in $L^2$,
\be\label{eq:initial.1}
\sup_{N\seg N_0}\sup_{\norm{\phi^N-\phi_{l_0}^N}_{\infty}<\rho}
\abs{\bbE_{\phi^N}\cro{\tau_{\e}^{N}(\cM_l)}-\bbE_{\phi_{l_0}^N}\cro{\tau_{\e}^{N}(\cM_l)}}
\ieg e^{\frac{\wh S-\eta}{\e}}.
\ee
For any sequence $\phi_i^N\in H^1$ of minima of $S^N$ converging to $\phi_i\in H^1$ in $L^2$, we also have
\be\label{eq:initial.2}
\sup_{N\seg N_0}\sup_{\norm{\phi_i^N-\phi^N}_{\infty}<\rho}
\abs{\bbP_{\phi_i^N}\cro{\tau^N_{\e}(\phi_{l_0})<\tau_{\e}^{N}(\cM_l)}
-\bbP_{\phi^N}\cro{\tau_{\e}^{N}(\phi_{l_0})<\tau_{\e}^{N}(\cM_l)}}
\ieg e^{-\frac{\eta}{\e}}.
\ee
\end{proposition}

The proof comes from a comparison between the deterministic process (i.e. $\e=0$) and the stochastic process starting from the moment of the hitting time .

\begin{proof}
Since the minima are not degenerate, we can assume $\rho$ small enough to get 
\be\label{eq:derive.1}
\bra{\frechet{S}{\phi}{}\phi,\phi-\phi_i}_{L^2}\ieg-b\norm{\phi-\phi_i}^2_{L^2}.
\ee
for some $b>0$, all $1<i<l$, and all $\phi\in \cB_{2\rho}(\phi_i)$.

First, let us prove similar estimates on the expectations of transition times for the localized process $u^N_R$. We denote by $\s^N(\phi^N)$ the hitting time $\tau_\e^{N,R}(\cM_l)$ for the process $u^N_R$ starting from $\phi^N$.
We set
\be\label{eq:loss.1}
\Om_{R}=\acc{\sup_{N\seg N_0}\sup_{\norm{u_0-\phi}_{\infty}<\rho}\norm{u^R_N(u_0)(t)-u^R_N(\phi)(t)}_{\infty}\ieg \rho e^{-mt},\forall t>0}.
\ee
From Proposition \ref{lem:contr}, we get $\bbP(\Om_{R})>1-e^{-C_R/\e}$.

Let us fix $\de_1>0$. We define $T(\e)=e^{\frac{\wh S-\de_1}{\e}}$ and we take $\e<\e_0$ such that $e^{-mT(\e)}<\rho$. On the set $\acc{\s^N(\phi_{l_0})>T(\e)}$, setting $\psi=u^{R}_N(\phi)(\s^N(\phi_{l_0}))$, we get
\be\label{eq:loss.3}
\norm{\psi-u^{R}_N(\phi_{l_0})(\s^N(\phi_{l_0}))}_{\infty}
<e^{-mT(\e)}<\rho
\ee
with probability at least $1-e^{-C_R/\e}$.
Let us suppose that $\s^N(\phi)-\s^N(\phi_{l_0})\seg0$ and that $u^{R}_N(\phi_{l_0})(\s^N(\phi_{l_0}))\in \cB_{\rho}(\phi_i)$. 

The deterministic process $u^{N,0}_R$ is the solution of \eqref{eq:mild.4} for the drift $b_R$ and $\e=0$. $\phi_i^N$ is a minimum of $S^N$, so $\phi_i^N$ is an equilibrium point of $u^{N,0}_R$.
Then using Equation \eqref{eq:derive.1}, we get for $t\seg0$
\be\label{eq:loss.4}
\norm{u^{N,0}_R(\psi)(t)-\phi_i}^2_{L^2}\ieg e^{-bt}\norm{\psi-\phi_i}^2_{L^2}\ieg e^{-bt}(e^{-mt}\rho+\rho)^2\ieg 4\rho^2e^{-bt}
\ee
by the triangle inequality.
For $t>t_0=\frac1{b}\ln(16)$, we obtain $\norm{u^{N,0}_R(\psi)(t)-\phi_i}_{L^2}\ieg \frac{\rho}{2}$.

From the large deviation principle, we can compare the deterministic solution with the perturbed one. We obtain $C>0$ such that
\be\label{eq:loss.5}
\bbP\cro{\acc{\norm{u^{N,0}_R(\psi^N)-u^{N}_R(\psi^N)}_{\infty,2t_0}<\frac{\rho}3}}
\seg 1-e^{-C/\e}.
\ee
Therefore, with probability at least $1-e^{-C/\e}-e^{-C_R/\e}$, we get 
$\norm{u^{N}_R(\psi)(2t_0)-\phi_i}_{L^2}<\frac{5\rho}{6}$ which implies
\be\label{eq:loss.7}
(\s^N(\phi)-\s^N(\phi_{l_0}))_+\ieg 2t_0.
\ee

We proceed similarly if $\s^N(\phi)-\s^N(\phi_{l_0})\ieg0$. In this case, we stop the process at $\s^N(\phi)$. Finally we get
$\abs{\s^N(\phi)-\s^N(\phi_{l_0})}\ieg 2t_0$ with probability at least $1-e^{-C'/\e}$, for some $C'>0$.

We obtain
\begin{align}\label{eq:loss.9}\nonumber
\bbE\cro{\abs{\s^N(\phi)-\s^N(\phi_{l_0})}}
&\ieg\bbE\cro{\abs{\s^N(\phi)-\s^N(\phi_{l_0})}\1_{\Om_{R}}\1_{\acc{\s^N(\phi_{l_0})>T(\e)}}}\\\nonumber
&\quad+\bbE\cro{\abs{\s^N(\phi)-\s^N(\phi_{l_0})}(\1_{\Om_{R}^c}+\1_{\acc{\s^N(\phi_{l_0})>T(\e)}^c})}\\
&\ieg2t_0(1-e^{-C'/\e})\bbP\cro{\Om_{R}\cap\acc{\s^N(\phi_{l_0})>T(\e)}}\\\nonumber
&\quad+\bbE\cro{\abs{\s^N(\phi)-\s^N(\phi_{l_0})}^2}^{\frac12}(\bbP[\Om_{R}^c]^{\frac12}+\bbP\cro{\acc{\s^N(\phi_{l_0})\ieg T(\e)}}^{\frac12}).
\end{align}
By using Proposition \ref{lem:contr}, we have $\bbP[\Om_{R}^c]<e^{-C_R/\e}$.
From Proposition \ref{lem:ldp}, we deduce that for $\e\ieg\e_0$
\be\label{eq:loss.11}
\bbP[\s^N(\phi_{l_0})\ieg T(\e)]<1-e^{-e^{-\frac{\de_1}{\e}}}<e^{-\frac{\de_1}{\e}}.
\ee
Moreover, we have for all $\de_2>0$
\be\label{eq:loss.12}
\bbE\cro{\abs{\s^N(\phi)-\s^N(\phi_{l_0})}^2}<e^{2\frac{\wh S+\de_2}{\e}}.
\ee
So we finally get
\begin{align}\label{eq:loss.13}
\bbE[\abs{\s^N(\phi)-\s^N(\phi_{l_0})}]
&\ieg2t_0(1-e^{-C_R/\e}-e^{-\frac{\de_1}{\e}})+e^{\frac{\wh S+\de_2}{\e}}(e^{-C/2\e}+e^{-\frac{\de_1}{2\e}})\ieg e^{\frac{\wh S-\eta}{\e}}.
\end{align}
By choosing $\de_1,\de_2$ and $\eta$ small enough, we prove the proposition for the localized process.

Let us now choose $R$ such that $\wh S(B^{\infty}_R(0),B^{\infty}_{\rho}(\phi_{l_0}))>\wh S+1$, then from Proposition \ref{lem:ldp}, we have 
\begin{align}\label{eq:loss.14}
\sup_{\phi\in B^{\infty}_{\rho}(\phi_{l_0})}\bbP_{\phi}[\tau_{\e}(B^{\infty}_R(0))\ieg \exp((\wh S+1-\de_3)/\e)=T_2(\e)]&\ieg e^{-C/\e}\\
\sup_{\phi\in B^{\infty}_{\rho}(\phi_{l_0})}\bbP_{\phi}[\tau_{\e}^{N}(\cM_l)\seg T_2(\e)]&\ieg e^{-C/\e}.
\end{align}
We consider the process $u$ starting from $\phi$ and $\phi_{l_0}$. Before $T_2(\e)$, with high probability, the processes are in $B^{\infty}_R(0)$ and coincide with $u_R$ up to this time. Moreover $T_2(\e)$ is much larger than the transition time, so the transition already occurs when the processes reach $B^{\infty}_R(0)^c$. Therefore, with very high probability, the transition time for the localized process is exactly the correct transition time.

For Equation \eqref{eq:initial.2}, we follow a similar method, by using Proposition \ref{lem:contr} for the localized process and then comparing the deterministic and stochastic processes in the neighborhood of a minimum.
\end{proof}

\section{Approximation of the potential}\label{sec:prefactor}

In this section, we prove (or refer to) results about the convergence of the potential and its related quantities. 

\subsection{Convergence of the potential}

Let us recall from Section \ref{ssec:contractivity} that for a point $u^N\in\bbR^N$, we denote also by $u^N$ the linear interpolation between the points $(x_i,u^N_i)$. For a function $u\in C_{bc}([0,1])$, we denote by $\wh u^N$ the linear interpolation between the points $(x_i,u(x_i))$.
We say that the sequence $u^N\in\bbR^N$ converges to $u\in H^1$ if the sequence of linear interpolations associated to $u^N$ (also denoted $u^N$) converges to $u$ in the $H^1$ norm.

Let us recall that $H S^N(u^N)$ is the Hessian matrix of $S^N$ at $u^N$ and can be interpreted as a bilinear form.
We prove the following proposition.
\begin{proposition}\label{prop:conv.potentiel}
For any sequence $u^N\in \bbR^N$ converging to $u\in H^1$, we have
\begin{itemize}
\item $S^N(u^N)\conv{N}{\infty}S(u)<\infty$
\item for any sequence $h^N$ converging to $h$: 
$\grad S^N(u^N)\cdot h^N\conv{N}{\infty}D_u S(h)$
\item for any sequences $h^N,k^N$ converging to $h,k$:
$$H S^N(u^N) (h^N,k^N)\conv{N}{\infty}D^2_u S(h,k).$$
\end{itemize}
If $u$ is twice differentiable $D_u S(h)=\int_0^1\frechet{S}{\phi}{}(u)h$ and if $k$ is twice differentiable $D^2_u S(h,k)=\int_0^1h\cH_uS k$.
\end{proposition}

\begin{proof}
Let $u^N\in \bbR^N$ be a sequence converging to $u\in H^1$, then $u^N$ converges uniformly on $[0,1]$ to $u$, so by dominated convergence,
\be\label{eq:conv.potentiel.4}
\frac1N\sum_{i=1}^NV(u^N_i)\conv{N}{\infty}\int_0^1V(u(x))\dx.
\ee
The convergence in $H^1$ directly ensures us that
\be\label{eq:conv.potentiel.5}
\frac1N\sum_{i=1}^NN^2(u^N_{i+1}-u^N_{i})^2=\int_0^1\abs{\pare{u^N}'(x)}^2\dx\conv{N}{\infty}\int_0^1\abs{u'(x)}^2\dx.
\ee

Let $h^N\in \bbR^N$ be some sequence converging to $h\in H^1$ then we have
\begin{align}\label{eq:conv.potentiel.6}\nonumber
\grad S^N(u^N)\cdot h^N
&=\sum_{i=1}^N\frac{\del S^N}{\del x_i}(u^N_i)h^N_i
=\frac1N\sum_{i=1}^N\g N^2(u^N_{i+1}-u^N_{i})(h^N_{i+1}-h^N_{i})+V'(u^N_i)h^N_i\\
&\conv{N}{\infty}\int_0^1\g u' h'+V'(u)h
\end{align}
by $L^2$ convergence of the derivatives and dominated convergence.
Lastly, the convergence of the Hessian is completely similar.
\end{proof}

\subsection{Convergence of the eigenvalues}

Let us consider a sequence of points $u^N\in\bbR^N$ converging to $u$ in $H^1$. We need to estimate the convergence of the eigenvalues $(N\la_{k,N})_{1\ieg k\ieg N}$ of $N\cdot HS^N(u^N)$ to the eigenvalues $(\la_{k})_{1\ieg k}$ of $\cH_u S$.

The convergence of a single eigenvalue $N\la_{k,N}$ for $k$ fixed, is obvious from Proposition \ref{prop:conv.potentiel}. The control of the convergence for all the eigenvalues is complex because of the higher eigenvalues (e.g. $\la_{N,N}$). This problem is closely related to the discrepancy between the eigenvalues of $\frac{\g}N\D^N$ and $\g\D$, the discrete Laplacian (defined by \eqref{eq:laplacian.1}) and the Laplacian. We denote $\la^0_{N,k}, \la^0_{k}$ their respective eigenvalues in the increasing order. For Dirichlet boundary conditions, we have
\be
e_{k,N}=N\la^0_{N,k}-\la^0_{k}=\g\cro{4N^2\sin^2\pare{\frac{k\pi}{2N}}-\pi^2k^2}.
\ee
Then $e_{N,N}=\g N^2(4-\pi^2)$ does not converge to $0$. The following proposition adapted from \cite{dehoog.anderssen01} gives us a control of the approximation of the eigenvalues and eigenvectors. 
\begin{proposition}\label{prop:conv.eigen}
Let us consider a sequence $u_N\in\bbR^N$ converging to $u\in C^2$ and such that
$\norm{u_N-u}_{\infty}=O\pare{\frac1{N^2}}$. 
We have:
\begin{itemize}
\item[(i)] there exist $\al\in[0,1[$ and a constant $C_1$ such that for all $N$ and $k<\al N$
\be\label{eq:conv.eigen.2}
\abs{N\la_{N,k}-\la_k-e_{k,N}}\ieg\frac{C_1}{N^2},
\ee
\item[(ii)] there exists a constant $C_2$ such that
$\abs{e_{N,k}}\ieg C_2k^4N^{-2}$,
\item[(iii)] for a fixed $k\ieg N$, the normalized (in $H^1$) eigenvector $\phi_{k,N}$ of $H S^N(u^N)$ associated to $\la_{k,N}$ converges in $H^1$ to the eigenvector $\phi_k$ of $\cH_{u} S$ associated to $\la_k$ and we have, for all $k$
\be\label{eq:conv.eigen.4}
\frac{\norm{\phi_{k,N}}_{\infty}}{\norm{\phi_{k,N}}_{2,N}}\ieg\frac{C}{\sqrt{N}}.
\ee
\end{itemize}
\end{proposition}

\begin{proof}
The proposition is an adaptation of the results of \cite{dehoog.anderssen01} in our case since $NHS^N(u^N)$ is the finite difference approximation of the Sturm-Liouville operator $\cH_u S$.
The original statement in \cite{dehoog.anderssen01} concerns an approximating sequence $u^N$ which is precisely the sequence $\wh{u^N}$ of linear interpolations of $u$. If we take a sequence $u^N$, then for all $y\in\bbR^N$
\begin{align}\label{eq:conv.eigen.5}
N\abs{HS^N(u^N)(y)-HS^N(\wh{u^N})(y)}
=\sum_{i=1}^N\abs{V''(u^N_i)-V''(u(x_i))}y_i^2
\ieg C\norm{u^N-u}_{\infty}\norm{y}_2^2.
\end{align}
Since $\norm{u_N-u}_{\infty}=O\pare{\frac1{N^2}}$, we deduce that the difference between the eigenvalues of $NHS^N(u^N)$ and $NHS^N(\wh{u^N})$ is bounded by $O(\frac1{N^2})$ which gives us the result. A similar control holds for the convergence of the eigenvectors. 
The last result \eqref{eq:conv.eigen.4} comes from the fact that for the eigenvectors of $\cH_u S$ (\cite{courant.hilbert1} pp.334-335), we have a constant $C$ such that $\norm{\phi_k}_{\infty}\ieg C\norm{\phi_k}_{L^2}$.
Then, since $\phi_{k,N}$ converges in $H^1$, it converges in $L^{\infty}$ and $L^2$, then the result comes from the fact that $\norm{\phi_{k,N}}_{2,N}\seg C\sqrt{N}\norm{\phi_{k,N}}_{L^2}$. 
\end{proof}

\begin{remark}
The normalized eigenvector $e_N=\frac{\phi_N}{\norm{\phi_N}_{2,N}}$ satisfies
\be\label{eq:conv.eigen.8}
\norm{e_N}^2_{\infty,N}
=\frac{\norm{\phi_N}^2_{\infty,N}}{\norm{\phi_N}^2_{2,N}}
\ieg \frac{\norm{\phi_N}^2_{L^\infty}}{N\norm{\phi_N}^2_{L^2}}
\ieg \frac{C}{N}\frac{\norm{\phi_N}^2_{H^1}}{\norm{\phi_N}^2_{L^2}}
\ieg\frac{C}{N}.
\ee
Thus, this proves that the coordinates of the normalized eigenvectors in $\bbR^N$ for the euclidean norm are uniformly bounded by $O\pare{\frac{1}{\sqrt{N}}}$.
\end{remark}

The following proposition from \cite{dehoog.anderssen01} states uniform estimates in the function $\phi$ of the eigenvalues of the Hessian operators $\cH_{\phi} S$ and $HS^N(\phi^N)$.
\begin{proposition}\label{prop:control.eigen}
Let $\phi_1^N,\phi_2^N$ be sequences converging in $H^1$ to $\phi_1, \phi_2$, then for all $N,k$
\begin{align}\label{eq:control.eigen.1}
\abs{\la^1_{k,N}-\la^2_{k,N}}&\ieg C&
\abs{\la^1_{k}-\la^2_{k}}&\ieg C
\end{align}
and
$\la^i_{k}=\pi^2k^2+\int_0^1V''(\phi_i(x))\dx+O\pare{\frac1{k^2}}$ for $i=1,2$.
\end{proposition}
\begin{remark}
This proposition shows the convergence of the infinite product of the ratio of eigenvalues denoted by $D(\phi,\psi)$
\begin{align}\label{eq:conv.eigen1.13}
\prod_{k=1}^{N}\frac{\la_k(\phi)}{\la_k(\psi)}
=\prod_{k=1}^{N}\cro{1+\frac{\la_k(\phi)-\la_k(\psi)}{\la_k(\psi)}}
\conv{N}{\infty}\prod_{k=1}^{\infty}\frac{\la_k(\phi)}{\la_k(\psi)}=D(\phi,\psi)
\end{align}
since
\be\label{eq:conv.eigen1.14}
\abs{\frac{\la_k(\phi)-\la_k(\psi)}{\la_k(\psi)}}\ieg\frac{C}{k^2}.
\ee
\end{remark}

\subsection{Product of eigenvalues}

We show the convergence of the product ratio of the eigenvalues of $HS^N(\phi^N)$ and $HS^N(\psi^N)$ to $D(\phi,\psi)$.
\begin{proposition}\label{prop:conv.eigen1}
For any $\phi^N,\psi^N$ converging in $H^1$ to $\phi, \psi$ such that $\cH S(\psi)$ and $\cH S(\phi)$ do not have a zero eigenvalue, and that
\begin{align}\label{eq:conv.eigen1.1}
\norm{\phi^N-\phi}_{\infty}\vee\norm{\psi^N-\psi}_{\infty}\ieg \frac{C}{N^2},
\end{align}
we have the convergence
\be\label{eq:conv.eigen1.2}
\frac{\det(HS^N(\phi^N))}{\det(HS^N(\psi^N))}
\conv{N}{\infty}D(\phi,\psi)=\prod_{k=1}^{+\infty}\frac{\la_k(\phi)}{\la_k(\psi)}.
\ee
\end{proposition}

\begin{proof}
The proof of the convergence comes from the fact that for small $k$ the approximated eigenvalues are close to the continuous ones ($\la_{k,N}\approx\la_k$) whereas this is not the case for $k$ close to $N$ (Proposition \ref{prop:conv.eigen}). The eigenvalues $\la_{k,N}(\phi)$, $\la_{k,N}(\psi)$ are close at the first order in $k$ uniformly on $\phi$, $\psi$ (Proposition \ref{prop:control.eigen}). Therefore we decompose the product in two parts for small $k$ (i.e. $k<\al N$ from Proposition \ref{prop:conv.eigen}) and large $k$.

Let us denote $\mu_{k,N}(\phi)=N\la_{N,k}(\phi^N)-\la_k(\phi)-e_{k,N}$.
From Proposition \ref{prop:conv.eigen}, there exists $0<\al<1$ such that, for $k\ieg \al N$, 
$\abs{\mu_{k,N}(\phi)}\ieg\frac{c}{N^2}$. The same holds for the sequence $\psi^N$.
Then, we get,
\begin{align}\label{eq:conv.eigen1.5}
\frac{N\la_{k,N}(\phi)}{\la_k(\phi)}\frac{\la_k(\psi)}{N\la_{k,N}(\psi)}
=\frac{1+\tht_{k,N}(\phi)}{1+\tht_{k,N}(\psi)}
=1+\frac{\tht_{k,N}(\phi)-\tht_{k,N}(\psi)}{1+\tht_{k,N}(\psi)}
\end{align}
where $\tht_{k,N}(\phi)=\la_{k}(\phi)^{-1}(e_{k,N}+\mu_{k,N}(\phi))$. Let us remark that for $k\ieg\al N$
\be\label{eq:conv.eigen1.6}
\abs{\tht_{k,N}(\psi)}
\ieg\frac{C}{k^2}\pare{\frac{k^4}{N^2}+\frac1{N^2}}
\ieg C\pare{\al^2+\frac1{N^2}}
\ee
thus if we take $\al$ small enough and $N$ large enough, we have $\abs{\tht_{k,N}(\psi)}<\frac12$. Hence we obtain
\begin{align}\label{eq:conv.eigen1.7}
\abs{\ln{\prod_{k=1}^{\al N}\frac{N\la_{k,N}(\phi)}{\la_k(\phi)}\frac{\la_k(\psi)}{N\la_{k,N}(\psi)}}}
\ieg2\sum_{k=1}^{\al N}\abs{\tht_{k,N}(\phi)-\tht_{k,N}(\psi)}
\ieg \frac{2C\al}{N}
\end{align}
since from Proposition \ref{prop:control.eigen}, 
$\abs{\tht_{k,N}(\phi)-\tht_{k,N}(\psi)}
\ieg
\frac{C}{N^2}$.

For $k>\al N$ we proceed similarly. Let us write
\begin{align}\label{eq:conv.eigen1.9}
\frac{N\la_{k,N}(\phi)}{\la_k(\phi)}\frac{\la_k(\psi)}{N\la_{k,N}(\psi)}
=\frac{1+\tht'_{k,N}}{1+\tht'_{k}}
=1+\frac{\tht'_{k,N}-\tht'_{k}}{1+\tht'_{k}}
\end{align}
where $\tht'_{k,N}=\la_{k,N}(\psi)^{-1}(\la_{k,N}(\phi)-\la_{k,N}(\psi))$ and alike for $\tht'_k$. From Proposition \ref{prop:control.eigen}, we get for all $k$ and $N>N_0$, that
$\abs{\tht'_{k,N}}\vee\abs{\tht'_{k}}\ieg \frac{C}{k^2}$.
Thus we obtain
\begin{align}\label{eq:conv.eigen1.11}
\abs{\ln{\prod_{k=\al N}^{N}\frac{N\la_{k,N}(\phi)}{\la_k(\phi)}\frac{\la_k(\psi)}{N\la_{k,N}(\psi)}}}
\ieg
\sum_{k=\al N}^N\frac{C}{k^2}\pare{1+\frac{C}{k^2}}
\ieg\frac{C}{N}
\end{align}
which finishes the proof.
\end{proof}

In fact, we need a slightly different convergence.
\begin{corollary}\label{cor:conv.eigen2}
Let be $\phi^N,\psi^N$ converging to $\phi,\psi$ such that
\begin{align}\label{eq:conv.eigen2.1}
\norm{\phi^N-\phi}_{L^2}\vee\norm{\psi^N-\psi}_{L^2}&\ieg \frac{C}{N}.
\end{align}
Then we have
\be\label{eq:conv.eigen2.2}
\frac{\det(HS^N(\phi^N))}{\det(HS^N(\psi^N))}
\conv{N}{+\infty}
D(\phi,\psi).
\ee
\end{corollary}

\begin{proof}
From the previous proposition, we get that
\be\label{eq:conv.eigen2.3}
\frac{\det(HS^N(\wh\phi^N))}{\det(HS^N(\wh\psi^N))}
\conv{N}{+\infty}
D(\phi,\psi)
\ee
where $\wh\phi^N$ (resp. $\wh\psi^N$) is the linear interpolation of $\phi$ (resp. $\psi$).
So we prove
\be\label{eq:conv.eigen2.4}
D_N=\frac{\det(HS^N(\wh\phi^N))}{\det(HS^N(\wh\psi^N))}
\cro{\frac{\det(HS^N(\phi^N))}{\det(HS^N(\psi^N))}}^{-1}
=\prod_{k=1}^N\frac{1+\tht_{k}(\phi)}{1+\tht_{k}(\psi)}
\conv{N}{\infty}1
\ee
where $\tht_{k}(\phi)=\la_{k,N}(\phi^N)^{-1}(\la_{k,N}(\wh\phi^N)-\la_{k,N}(\phi^N))$.
From the fact that
$\norm{\phi^N-\phi}_{L^2}\ieg \frac{C}{N}$
we obtain
$\norm{\phi^N-\wh{\phi^N}}_{L^2}\ieg \frac{C'}{N}$. Then for all $y\in \bbR^N$, we have
\begin{multline}\label{eq:conv.eigen2.7}\nonumber
\abs{HS^N(\phi^N)(y)-HS^N(\wh\phi^N)(y)}
=\frac{1}{N}\sum_{i=1}^N\abs{V''(\phi^N_i)-V''(\phi(x_i))}\abs{y_i}^2\\
\ieg \frac{C}{N}\sum_{i=1}^N\abs{\phi^N_i-\phi(x_i)}\abs{y_i}^2
\ieg \frac{C}{\sqrt{N}}\norm{\phi^N-\wh\phi^N}_{L^2}\norm{y}_{4,N}^2
\ieg \frac{C}{N^{3/2}}\norm{y}_{2,N}^2.
\end{multline}
Therefore we get that $\abs{\la_{k,N}(\phi^N)-\la_{k,N}(\wh\phi^N)}\ieg\frac{C}{N^{3/2}}$.
The same holds for $\psi$.

Then, we obtain
\be\label{eq:conv.eigen2.10}
\abs{\tht_{k}(\psi)}
\ieg\frac{CN}{k^2}\times\frac{1}{N^{3/2}}
\ieg\frac{C}{k^2\sqrt{N}}\ieg\frac12
\ee
for $N$ sufficiently large.

Thus we get
\begin{align}\label{eq:conv.eigen2.11}
\abs{\ln\cro{D_N}}
\ieg
\sum_{k=1}^N\frac{\abs{\tht_{k}(\phi)-\tht_{k}(\psi)}}{1+\tht_{k}(\psi)}
\ieg
2\sum_{k=1}^N\abs{\tht_{k}(\phi)}+\abs{\tht_{k}(\psi)}
\ieg
4C\sum_{k=1}^N\frac{1}{k^2\sqrt{N}}.
\end{align}
Then let us fix $\eta>0$, we have
\begin{align}\label{eq:conv.eigen2.12}
\abs{\ln\cro{D_N}}
\ieg
C\sum_{k=1}^{\eta\sqrt[3]{N}}\frac{1}{k^2\sqrt{N}}
+C\sum_{k=\eta \sqrt[3]{N}}^{ N}\frac{1}{k^2\sqrt{N}}
\ieg C\eta N^{-1/6}+\frac{C}{\eta^2}N^{-1/6}.
\end{align}
Therefore we get
$\limsup_{N\to\infty}\abs{\ln\cro{D_N}}=0$
which proves the proposition.
\end{proof}

\subsection{Approximated stationary points}\label{ssec:last}

The last property we need to check is that for each stationary point of $S$, there exists a unique sequence of stationary points of $S^N$ converging to this stationary point. Moreover, to ensure the limit of the ratio of eigenvalues, this convergence has to be fast enough (see Corollary \ref{cor:conv.eigen2}).
To this aim, we have the following proposition.
\begin{proposition}\label{prop:stationary}
There exist $C,N_0$, such that for all $N>N_0$, there is for each minimum (resp. saddle point) $\phi$ of $S$ a minimum (resp. saddle point) $\phi^N$ of $S^N$ such that
\begin{align}\label{eq:stationary.1b}
\norm{\phi-\phi^N}_{L^2}&\ieg \frac{C}{N}
\end{align}
where $\wh{\phi^N}$ is the linear interpolation of $\phi$.
\end{proposition}

\begin{proof}
Since by Assumption \ref{hyp:S}, there is a finite number of saddles and stationnary points then we only need to prove the proposition for a given saddle or minimum. Let $\phi$ be a minimum, we prove that there is sequence $\phi^N$ of minima of $S^N$ such that
\be\label{eq:stationary.2b}
\norm{\phi^N-\wh{\phi^N}}_{L^2}\ieg \frac{C}{N}.
\ee
The result \eqref{eq:stationary.1b} follows from \eqref{eq:stationary.2b} since we already have that 
\be
\norm{\phi-\wh{\phi^N}}_{L^2}\ieg\norm{\phi-\wh{\phi^N}}_{\infty}\ieg \frac{C}{N^2}.
\ee

In order to prove \eqref{eq:stationary.2b}, we use a fixed point theorem. Let us consider the ball $B_{C/\sqrt{N}}$ of radius $\frac{C}{\sqrt{N}}$ in the $\norm{\cdot}_{2,N}$ norm where $C$ is a constant we will fix later. We want to find  $z^0\in B_{C/\sqrt{N}}$ such that $\grad S^N(\wh{\phi^N}+z^0)=0$.
In that case we will have $\phi^N=\wh\phi^N+z$ and 
\be\label{eq:stationary.6}
\norm{\phi^N-\wh{\phi^N}}^2_{L^2}\ieg\frac1{N}\norm{z}^2_{2,N}\ieg\frac{C}{N^2}.
\ee
By a Taylor expansion of the gradient we have
\begin{align}\label{eq:stationary.7}
\grad S^N(\wh{\phi^N}+z)_i
&=\grad S^N(\wh{\phi^N})_i+(HS^N(\wh{\phi^N})z)_i+g_i(z)
\end{align}
where $g_i$ is the remainder which can take the form
\be\label{eq:stationary.8}
g_i(z)=\int_0^1(1-t)\frac{\del^3S^N}{\del z_i^3}(\wh{\phi^N}+tz)z_i^2\dt
=\frac1N\int_0^1(1-t)V'''(\phi_i+t z_i)z_i^2\dt.
\ee
Then we have for all $z,y\in B_{C/\sqrt{N}}$
\begin{align}\label{eq:stationary.9}
\abs{g_i(z)}&\ieg \frac{C_0}{N} z_i^2&\text{ and }&&
\abs{g_i(x)-g_i(y)}&\ieg \frac{C_0}{N}\abs{z_i^2-y_i^2}\ieg\frac{2C_0}{N^{3/2}}\abs{z_i-y_i}.
\end{align}
Let us also remark that since $\phi$ is a stationary point for the potential $S$, thus we have $-\g \phi''(x_i)+V'(\phi(x_i))=0$. Therefore we get
\begin{align}\label{eq:stationary.10}\nonumber
\abs{\grad S^N(\wh{\phi^N})_i}
&=\abs{\grad S^N(\wh{\phi^N})_i-\frac{1}N\pare{-\g \phi''(x_i)+V''(\phi(x_i))}}\\
&=\frac1N\abs{\g N^2(\phi(x_{i+1})-2\phi(x_i)+\phi(x_{i-1}))-\g\phi''(x_i)}
\ieg\frac{C_1}{N^2}.
\end{align}
For $N$ sufficiently large $HS^N(\wh{\phi^N})$ is not degenerate then $z^0$ is solution of the fixed point equation
\be\label{eq:stationary.11}
z^0=HS^N(\wh{\phi^N})^{-1}(-\grad S^N(\wh{\phi^N})-g_i(z^0))=F(z^0).
\ee
The $(2,N)$-norm of $HS^N(\wh{\phi^N})^{-1}$ is bounded by the inverse of the smallest eigenvalue (in absolute value). Then
$\norm{HS^N(\wh{\phi^N})^{-1}}_{2,N}\ieg C_2 N$. For $z\in B_{C/\sqrt{N}}$, we get
\begin{align}\label{eq:stationary.13}\nonumber
\norm{F(z)}^2_{2,N}
&\ieg \norm{HS^N(\wh{\phi^N})^{-1}}^2_2
\pare{\norm{\grad S^N(\wh{\phi^N})}^2_2+\sum_{i=1}^N\abs{g_i(z)}^2}
\ieg C^2_2N^2\pare{\frac{C^2_1}{N^3}+C^2_0\norm{z}_{4,N}^4}\\
&\ieg C'_1\pare{\frac1N+N^2\norm{z}_{2,N}^4}
\ieg C'_1\pare{\frac1N+\frac {C^4}{N^2}}
\ieg\frac{C^2}{N}
\end{align}
for $C$ sufficiently small. Therefore
$F(B_{C/\sqrt{N}})\subset B_{C/\sqrt{N}}$.
We also have for $z,y\in B_{C/\sqrt{N}}$,
$F(y)-F(z)=HS^N(\wh{\phi^N})^{-1}(-g_i(y)+g_i(z))$.

Then
\begin{align}\label{eq:stationary.16}\nonumber
\norm{F(y)-F(z)}_{2,N}^2
&\ieg C_2N^2\sum_{i=1}^N\abs{-g_i(y)+g_i(z)}^2
\ieg \frac{C'_2}{N}\norm{y-z}_{2,N}^2.
\end{align}
Thus $F$ is a contraction for $N$ sufficiently large. By the fixed point Theorem, there exists $z^0\in B_{C/\sqrt{N}}$ solution of $z^0=F(z^0)$ which proves Proposition \ref{prop:stationary}
\end{proof}

\section{Estimates}\label{sec:estimates}

\subsection{Description}\label{ssec:description}

In this section, we compute uniformly in the dimension the expectation of the transition times. We proceed as in \cite{bbm10} and use the potential theory developed in \cite{bovier04}.
Let us consider the $N$-dimensional diffusion
\be\label{finite.2c}
\dint Y_t=-\grad S_N(Y_t)\dt +\sqrt{2\e}\dint B_t
\ee
which comes from \eqref{eq:finite.2b} with the time change $Y_{h_Nt}=X_t$.
We denote by $\mu^N$ the invariant measure for the process $Y$
\be\label{eq:estimate.1}
\mu^N(\dx)=e^{-\frac{S^N(x)}{\e}}\dx.
\ee

Let us consider the norms for $y\in\bbR^N$ and $p\seg 1$
\begin{align}\label{eq:norm.3b}
\norm{y}_{p,N}^p&=\sum_{x=1}^N\abs{y_i}^p&
\norm{y}_{\infty,N}&=\max_{i=1\cdots N}\abs{y_i}.
\end{align}

\begin{remark}
As in the previous section, we associate to a point $y\in \bbR^N$ its linear interpolation on $[0,1]$ between the points $(x_i,y_i)$ ($x_i$ is given by \eqref{eq:points.d},\eqref{eq:points.n}) that we denote by $y$.
Let us consider the $L^p$ norm of $y$ on $[0,1]$, we have for all $p\in[1,+\infty]$
\begin{align}\label{eq:norm.4}
\frac1{(4N)^{1/p}}\norm{y}_{p,N}\ieg\norm{y}_{L^p}
=\cro{\int_0^1\abs{y(x)}^p\dx}^{\frac1p}
\ieg\frac1{N^{1/p}}\norm{y}_{p,N}.
\end{align}
This can be done using the Riesz-Thorin Theorem, remarking that
\begin{align}\label{eq:norm.5}
\frac1{4N}\norm{y}_{1,N}
&\ieg\norm{y}_{L^1}
\ieg\frac1{N}\norm{y}_{1,N}&\text{ and }&&
\norm{y}_{\infty,N}&=\norm{y}_{L^{\infty}}.
\end{align}
\end{remark}

In order to introduce the other norms, we need the following a priori estimates on the eigenvalues of the Hessian of $S^N$. Let us recall the Hessian of $S^N$ at a point $\phi^N\in\bbR^N$ is
\be\label{eq:eigen.1}
HS^N(\phi^N)(h)_j=-\frac1N(\D^Nh)_j+\frac1NV''(\phi^N(x_j))h_{j},\text{ for $h\in\bbR^N$}
\ee
with the suitable boundary conditions.

\begin{lemma}[\cite{dehoog.anderssen01}]\label{lem:eigenvaluesN}
For all $\phi^N\in \bbR^N$ such that $\norm{\phi^N}_{\infty}<A$, the eigenvalues $(\la_{k,N}(\phi^N))_{k=1}^N$ of $HS^N(\phi^N)$ arranged in increasing order
satisfy the bound
\be\label{eq:eigen.2}
m(A)k^2-1\ieg N\la_{k,N}(\phi^N)\ieg M(A)k^2+1
\ee
where $m(A)$ and $M(A)$ do not depend on $N$ and $\phi^N$ (only on $A$).
\end{lemma}

Let us fix $\phi^N\in \bbR^N$. We consider the orthonormal eigenvectors $(v_{l})_l$ of $HS^N(\phi^N)$. The decomposition of $h\in \bbR^N$ in this orthonormal basis is given by $h=\sum_{l=1}^N\wt h_iv_l$.
For $p\in[1,\infty]$, we define the norms $\norm{h}_{p,\cF}$
\begin{align}\label{eq:norm.7}
\norm{h}^p_{p,\cF}&=\sum_{i=1}^N\abs{\wt h_i}^p&\norm{h}_{\infty,\cF}&=\max_{i=1\cdots N}\abs{\wt h_i}.
\end{align}
As in \cite{bbm10}, these are the norms we use to control the approximations of the potential around our stationary points. Let us note that the norms depend on the point $\phi^N$.

\begin{remark}\label{rem:norm}
As in Section 4.1.1 in \cite{bbm10}, the Hausdorff-Young Theorem can be adapted to the norms $\norm{\cdot}_{p,\cF}$ and $\norm{\cdot}_{p,N}$. For all $2\ieg p\ieg +\infty$ and $q$ such that $q^{-1}+p^{-1}=1$, we obtain
\be\label{eq:norm.12}
\frac1{N}\norm{x}^{p}_{p,N}\ieg C\pare{\frac{1}{\sqrt{N}}\norm{x}_{q,\cF}}^p.
\ee
In fact, let $T:\bbR^N\to\bbR^N$ be the linear mapping $T(y)=\sum_{k=0}^{N-1}y_kv_{N,l}(z^*_i)$. By definition, $\norm{Ty}_{p,\cF}=\norm{y}_{p,N}$. The proof of \eqref{eq:norm.12} is an application of the Riesz-Thorin Theorem, between $p=2$ and $p=\infty$.
On one hand, we have $\norm{Ty}^{2}_{2,N}=\norm{y}^{2}_{2,N}$ since the eigenvectors form a orthonormal basis. On the other hand, we have $\norm{Ty}_{\infty,N}\ieg \frac{C}{\sqrt{N}}\norm{y}_{1,N}$ since the coordinates of the eigenvectors of the basis are bounded by $\frac{C}{\sqrt{N}}$ (see Lemma \ref{prop:conv.eigen}, Equation \eqref{eq:conv.eigen.4}).
\end{remark}

\bigskip

Let us recall the infinite dimensional situation. The process $u$ starts from a minimum $\phi_{l_0}$ of $S$ and reaches the set of minima $\cM_l$. We denote by $\wh S_0=\wh S(\phi_{l_0},\cM_l)$ the height of the saddle points defined by \eqref{eq:saddle.1}.

By Assumption \ref{hyp:S}, for all $N$ sufficiently large, we have a finite set $\cM^N=\{x^*_i\}$ of minima of $S^N$. From Proposition \ref{prop:conv.eigen} and Proposition \ref{prop:stationary}, we deduce that a sequence of minima $x^*_{l_0}$ converges to $\phi_{l_0}$. Similarly, there is a subset $\cM_l^N$ of $\cM^N$ such that each minimum of $\cM_l^N$ converges to a minimum of $\cM_l$. 

We construct a graph for the finite dimensional case as the graph for the infinite dimensional case in Section \ref{ssec:main}. The vertices are the minima $\cM^N$. The edges are the saddle points $z^*_k$ of $S^N$ for which $\abs{\wh S_0-S^N(z^*_k)}<\eta$ for some fixed $\eta>0$. We connect the edge $z^*_k$ between the two minima that the saddle point $z^*_k$ connects directly. To each saddle point $z_k^*$, we associate a weight 
\be\label{eq:weight.1}
w^*_k=\frac{\abs{\la_N^-(z_k^*)}e^{-\frac{S^N(z_k^*)}{\e}}}{\sqrt{\abs{\det{HS^N(z_k^*)}}}}.
\ee
To each minima $x^*_j$, we associate a value $a_j=a(x^*_j)\in\bbR$. We denote by $a_{i+}$ and $a_{i-}$ the two values associated to the minima connected by the saddle point $z^*_i$.

We associate to this graph a quadratic form $Q^N(a)$, for $a$ a real vector indexed by the minima $\cM^N$
\be\label{eq:quad.1}
Q^N(a)=\sum_{z^*_l}w^*_l(a_{l+}-a_{l-})^2.
\ee
The equivalent conductance, $C^*(N,\e)$, between the sets $x_{l_0}^*$ and $\cM^N_{l}$ is defined by
\be\label{eq:capacity.1}
C^*(N,\e)=\inf\acc{Q^N(a), a(x_{l_0}^*)=1, a(x_i^*)=0, x_i^*\in\cM^N_{l}}.
\ee

We recall the fundamental formula \eqref{eq:estimate.2} proved in \cite{bovier04}. 
The expression of the expectation of the hitting time $\tau^N_{\e}(\cB_{\rho}^N(x_{l_0}^*))$ is based on two quantities: the equilibrium potential and the capacity with respect to the sets $\cB_{\rho}^N(x_{l_0}^*)$ and $\cB_{\rho}^N(\cM^N_{l})$.
The equilibrium potential, $h^*$, is defined by $h^*(x)=\bbP_x[\tau^N_{\e}(\cB_{\rho}^N(x_{l_0}^*))<\tau^N_{\e}(\cB_{\rho}^N(\cM^N_{l}))]$. The Dirichlet form,  $\ccE^N$,  associated with the diffusion process $Y$ on $\bbR^N$ is
\be\label{eq:estimate.4}
\ccE^N(h)=\e\int_{\bbR^N}\norm{\grad h(x)}^2_{2,N}\mu^N(\dx).
\ee
The capacity is the evaluation of the Dirichlet form on $h^*$. The capacity also satisfies a variational principle. We have
\begin{multline}\label{eq:estimate.3}
\capa{\cB_{\rho}^N(x_{l_0}^*),\cB_{\rho}^N(\cM^N_{l})}=\ccE^N(h^*)\\
=\inf\acc{\ccE^N(h), h\in H^1(\bbR^N), h=1 \text{ on $\cB_{\rho}^N(x_{l_0}^*)$},
h=0 \text{ on $\cB_{\rho}^N(\cM^N_{l})$}}.
\end{multline}

The expectation of the hitting time is expressed by
\be\label{eq:estimate.2}
\bbE_{\nu^{N}}[\tau^N_{\e}(\cB_{\rho}^N(\cM^N_{l}))]=\frac{\int_{\bbR^N}h^*(x)\dint \mu^N(x)}{\capa{\cB_{\rho}^N(x^*_{l_0}),\cB_{\rho}^N(\cM^N_{l})}}
\ee
where $\nu^N$ is a probability measure on $\del \cB^N_{\rho}(x^*_{l_0})$.

\subsection{Capacity}\label{ssec:capacity}

We prove that the capacity defined in \eqref{eq:estimate.3} can be estimated by the equivalent conductance $C^*(N,\e)$ defined in \eqref{eq:capacity.1}.
\begin{proposition}\label{prop:capa}
For all $\e<\e_0$ and $\rho$, we have
\be\label{eq:cap.1}
\capa{\cB_{\rho}^N(x_{l_0}^*),\cB_{\rho}^N(\cM^N_{l})}
= \e \sqrt{2\pi\e}^{N-2}C^*(N,\e)(1+\psi_1(\e,N))
\ee
where $\limsup_{N\to+\infty}\abs{\psi_1(\e,N)}<\sqrt{\e}\abs{\ln(\e)}^{3/2}$ for all $N>N_0$.
\end{proposition}

The proof of this result is an adaptation to the case of a finite number of saddle points of Proposition 4.3 in \cite{bbm10}. The estimate of the capacity is made in two steps: an upper bound and a lower bound.

\subsubsection{Upper bound}\label{sssec:upper}

We have the following proposition.
\begin{proposition}\label{prop:upper}
For all $\e<\e_0$ and $\rho$, we have
\be\label{eq:upper.1}
\capa{\cB_{\rho}^N(x_{l_0}^*),\cB_{\rho}^N(\cM^N_{l})}
\ieg\e \sqrt{2\pi\e}^{N-2}C^*(N,\e)(1+\psi_u(\e,N))
\ee
where $\limsup_{N\to\infty}\abs{\psi_u(\e,N)}<\sqrt{\e}\abs{\ln(\e)}^{3/2}$.
\end{proposition}

\begin{proof}
The proof of this upper bound follows the proof of Lemma 4.4 in \cite{bbm10}.
To obtain an upper bound for the capacity, we just estimate the Dirichlet form on a test function $h^+$. $h^+$ is defined on some neighborhood $C^N_{\de}(z^*_i)$ of each saddle point $z^*_i$ for some $\de>0$ small enough.

In the local orthonormal basis (given by coordinates $y^{(i)}\in\bbR^N$) of the saddle point $z^*_i$, the neighborhood $C^N_{\de}(z^*_i)$ is defined by
\be\label{eq:approx.1}
C^N_{\de}(z^*_i)=\acc{y^{(i)}\in
\bbR^N:\,|y^{(i)}_l|\ieg \de\frac{r_{l}}{\sqrt{\abs{\la_{N,l}}}},\, 0\ieg
l\ieg N-1}+z^*_i
\ee
where $(r_l)$ is a sequence satifying $\sum_{l}\frac{r_l^{3/2}}{l^{3/2}}<\infty$ and $(\la_{N,l})_l$ are the eigenvalues in the increasing order of $HS^N(z^*_i)$.
Let us denote $C^N_{\de}=\cup_{i}C^N_{\de}(z^*_i)$.

Let us consider
\be\label{eq:upper.4}
S_{N,\de}=\acc{x, S^N(x)\seg S^N(z^*_i)+c\de^2, \forall i}.
\ee
The set $(S_{N,\de}\cup C^N_{\de})^c$ contains a finite number of connected components denoted $D_j$ since each of them contains at least a minimum $x^*_j$ (which are in finite number by Assumption \ref{hyp:S}). For each connected component $D_j$, we define $h^+$  to be the constant $a_j\in[0,1]$.
For a saddle $z^*_i$, we denote $D_{i+}$ and $D_{i-}$ the connected components attained from $z^*_i$ when $y^{(i)}=(\de\s_0,0)$ and $y^{(i)}=(-\de\s_0,0)$ respectively.

On $S_{N,\de}\setminus C^N_{\de}$, we take $h^+$ of class $C^1$ and such that $\norm{\nabla h^+}_{2,N}\leq \frac{c_1}{\de}$. Then we define $h^+$  on each $C^N_{\de}(z^*_i)$ in the local coordinates, by
$h^+(y^{(i)})=f_i(y^{(i)}_0)$
where
\be\label{eq:upper.5b}
f_i(y_0)=(a^{i-}-a^{i+})\frac{\int_{y_0}^{\de\s_0} e^{-\abs{\la_{N,0}}t^2/2\e}\dt}
{\int_{-\de\s_0}^{\de\s_0}e^{-\abs{\la_{N,0}}t^2/2\e}\dt}+a^{i+}.
\ee

Therefore, we have to estimate $\cE^N(h^+)=\sum_{i}I_1(i)+I_2$ with
\begin{align}\label{eq:upper.7}
I_1(i)&=\e\int_{C^N_{\de}(z^*_i)}\norm{\grad h^+(x)}^{2}_{2,N}e^{-\frac{S^N(x)}{\e}}\dx,&
I_2&=\e\int_{S_{N,\de}\setminus B^N_{\de}}\norm{\grad h^+(x)}^2_{2,N}e^{-\frac{S^N(x)}{\e}}\dx.
\end{align}

Taking $\de=K\sqrt{\e\abs{\ln\e}}$, the integrals $I_1(i)$ give us the right asymptotics and are estimated by an adaptation of Lemma 4.4 from \cite{bbm10}. The quadratic approximation of the potential on the sets $C^N_{\de}(z^*_i)$ is a consequence of Remark \ref{rem:norm} and of the choice of the sets $C^N_{\de}(z^*_i)$.
The integral $I_2$ is computed by following the same method as in Lemma 4.6 in \cite{bbm10}.

Therefore, we obtain that for all $(a_j)_j$, for $N\seg N_0(\e)$
\begin{align}\label{eq:upper.4b}\nonumber
\capa{\cB_{\rho}^N(x^*),\cB_{\rho}^N(\cM^N_{l})}
&\ieg\sum_i\e \sqrt{2\pi\e}^{N-2}
\frac{(a^{i-}-a^{i+})^2\abs{\la_{N,0}}e^{-\frac{S^N(z^*_i)}{\e}}}{\sqrt{|\det(H S^N(z^*_i))|}}
(1+A_1\sqrt{\e}\abs{\ln(\e)}^{3/2}).
\end{align}
Taking the minimum of the right-hand side over $a$, we get the result \eqref{eq:upper.1}.
\end{proof}

\subsubsection{Lower bound}\label{sssec:lower}

We now prove the corresponding lower bound.
\begin{proposition}\label{prop:lower}
For all $\e<\e_0$ and $\rho$, we have
\be\label{eq:lower.1}
\capa{\cB_{\rho}^N(x^*),\cB_{\rho}^N(\cM^N_{l})}
\seg \e \sqrt{2\pi\e}^{N-2}C^*(N,\e)(1+\psi_l(\e,N))
\ee
where $\limsup_{N\to\infty}\abs{\psi_l(\e,N)}<\sqrt{\e}\abs{\ln(\e)}^{3/2}$.
\end{proposition}

\begin{proof}
The proof is adapted from \cite{bbm10}. For a saddle point $z^*_i$, we take a narrow corridor from one (local) minimum to another one and minimize the Dirichlet form on the union of these corridors. In \cite{bbm10}, this corridor was a rectangle because of the particular case considered. In this article, we have to be more precise about their construction.
We use the same notations as in the proof of the upper bound. 

Let us fix $\de_0$. We consider the subset of $\bbR^{N-1}$
\be\label{eq:lower.2}
C^{N,\bot}_{\de}(z^*_i)=\acc{y^{(i)}\in
\bbR^N:\,|y^{(i)}_l|\ieg \de\frac{r_{l}}{\sqrt{\abs{\la_{N,l}}}},\, 1\ieg
l\ieg N-1}.
\ee
and we define $C^{N}_{\de}(z^*_i)=[-\de_0,\de_0]\times C^{N,\bot}_{\de}(z^*_i)+z^*_i$.
We denote by $x^*_{i-}$ and $x^*_{i+}$ the two minima of the basins surrounding $z^*_i$.

Let $(\g_0(s))_{s\in[-s_-,s_+]}$ be a regular $C^2$ path from $x_{i-}$ to $x_{i+}$ with $\g_0(s)=z^*_i+(s,0)$ for $s\in[-\de_0,\de_0]$.
We also suppose that there is $\eta>0$ for which $S^{N}(\g_0(s))<S^N_0-3\eta$ for $\abs{s}\seg\de_0$ and that $\norm{\g_0'(s)}_{2,N}=1$.
Let, for all $s$, $A(s)$ be an isomorphism from $\bbR^{N-1}$ to $\g'_0(s)^{\bot}\subset\bbR^N$ of class $C^1$ in $s$ and such that for $\abs{s}<\de_0$, $A(s)y=(0,y_1,\cdots,y_{N-1})$. Then we construct a family of paths $\g(s,y_{\bot})$ by
\be\label{eq:lower.4}
\g(s,y_{\bot})=\g_0(s)+A(s)y_{\bot}.
\ee
Such a construction of a path $\g_0$ is always possible in the infinite dimensional setting (because of Assumption \ref{hyp:S}). Then taking the finite dimensional projection, it gives us a path for the finite dimensional case.

We define the corridor from $x_{i-}$ to $x_{i+}$, for $\de>0$ small enough
\begin{align}\label{eq:lower.5}
C_{\de}(z^*_i)
&=\acc{x=\g(s,y_{\bot}),y_{\bot}\in C^{N,\bot}_{\de}(z^*_i), \forall s}.
\end{align}
Let $h$ be the equilibrium potential which realizes the minimum of the Dirichlet form and define $a^{i\pm}(y_{\bot})=h(x_{i\pm}+A(\pm s_\pm)y_{\bot})$, the values near the minimum.
\smallskip

To estimate a lower bound, we are going to restrict the Dirichlet form on the union of the corridors $C_{\de}(z^*_i)$:
\begin{align}\label{eq:lower.6}
\cE^N(h)
&=\e\int_{\bbR^N}\norm{\grad h}^2_{2,N}\mu^{N}(\dx)
\seg\sum_{i}\e\int_{C_{\de}(z^*_i)}\norm{\grad h}^2_{2,N}\mu^{N}(\dx)
=\e\sum I_5(i).
\end{align}

We define the function $f_i$ on $C_{\de}(z^*_i)$, by $f_i(s,y_{\bot})=h(\g(s,y_\bot))$. The change of variable on $C_{\de}(z^*_i)$ gives us the Jacobian $g_i(s,y_{\bot})=\det(J\g)(s,y_\bot)$ and we obtain
\begin{align}\label{eq:lower.7}
I_5(i)
&\seg
\int_{B^{N,\bot}_{\de}(z^*_i)}\int_{-s_-}^{s_+}\abs{\frac{\del f_i}{\del s}}^2e^{-S^N(\g(s,y_{\bot}))/\e}g_i(s,y_{\bot})\ds\dy_{\bot}.
\end{align}
We take $y_{\bot}$ as a parameter then the second term is bounded below by the minimum over functions $f_i$ of the integral
\be\label{eq:lower.8}
\int_{-s_-}^{s_+}\abs{\frac{\del f_i}{\del s}}^2e^{-S^N(\g(s,y_{\bot}))/\e}g_i(s,y_{\bot})\ds
\ee
with the conditions $f_i(-s_-,y_{\bot})=h(x_{i-}+A(-s_-)y_{\bot})=a^{i-}(y_{\bot})$ and $f_i(s_+,y_{\bot})=h(x_{i+}+A(s_+)y_{\bot})=a^{i+}(y_{\bot})$. This gives us a lower bound for the capacity.

A simple computation shows that the function $f_i$ realizing this lower bound is
\be\label{eq:lower.9}
f_i(s,y_\bot)=(a^{i+}(y_{\bot})-a^{i-}(y_{\bot}))\frac{\int_{-s_-}^{s}e^{S^N(s,y_{\bot})/\e}g_i(s,y_{\bot})^{-1}\ds}
{\int_{-s_-}^{s_+}e^{S^{N}(s,y_{\bot})/\e}g_i(s,y_{\bot})^{-1}\ds}+a^{i-}(y_{\bot}).
\ee
Inserting this function in the integral \eqref{eq:lower.7}, we obtain
\be\label{eq:lower.10}
I_5(i)
\seg 
\int_{C^{N,\bot}_{\de}(z^*_i)}(a^{i+}(y_{\bot})-a^{i-}(y_{\bot}))^2\cro{\int_{-s_-}^{s_+}e^{S^{N}(s,y_{\bot})/\e}g_i(s,y_{\bot})^{-1}\ds}^{-1}\dy_{\bot}.
\ee
The end of the proof comes from an upper bound of the integral uniformly for $y_{\bot}\in C^{N,\bot}_{\de}(z^*_i)$. We write
\be\label{eq:lower.11}
\int_{-s_-}^{s_+}e^{S^{N}(s,y_{\bot})/\e}g_i(s,y_{\bot})^{-1}\ds
= I_6(i)+I_7(i)
\ee
where
\begin{align}\label{eq:lower.12}
I_6(i)
&=\int_{-\de_0}^{\de_0}e^{S^{N}(s,y_{\bot})/\e}g_i(s,y_{\bot})^{-1}\ds
&\text{and}&&
I_7(i)
&=\int_{\abs{s}>\de_0}e^{S^{N}(s,y_{\bot})/\e}g_i(s,y_{\bot})^{-1}\ds.
\end{align}

As in Lemma 4.8 in \cite{bbm10}, we control the quadratic approximation near the saddle $z^*_i$ with the following lemma for which we omit the proof.
\begin{lemma}\label{lem:approx.low}
For all $y=(s,y_{\bot})\in C^{N}_{\de}(z^*_i)$, if the sequence $(r_l)_l$ satisfies $\sum_{l}\frac{r_l^{3/2}}{l^{3/2}}<\infty$, we have for $\de_0\seg\de$
\begin{align}\label{eq:approx.low.1}
\abs{S^{N}(\g(s,y_{\bot})+z^*_i)-S^{N}(\g(0,y_{\bot})+z^*_i)+\frac1{2}\abs{\la_{0,N}}s^2}&\ieg A_6\de_0^3\\\label{eq:approx.low.2}
\abs{S^N(z^*_i+\g(0,y_{\bot}))-S^N(z^*_i)-\frac12\sum_{k=1}^{N-1}\la_{N,k}y_k^2}&<A_8\de^3.
\end{align}
\end{lemma}

Following the proof of Lemma 4.7 in \cite{bbm10}, we can also prove the existence of a constant $A_6$ such that for all $N$ and $y_{\bot}$
\be\label{eq:lower1.1}
I_6(i)\ieg e^{\frac{S^N(z^*_i+(0,y_{\bot}))}{\e}}
\sqrt{\frac{2\pi\e}{\abs{\la_{N,0}}}}\pare{1+A_6\frac{\de_0^{3}}{\e}}.
\ee
%

In addition, we need to prove an upper bound for the integral $I_7(i)$.
\begin{lemma}\label{lem:lower2}
There exists a constant $A_7$ such that for all $N$ and $y_{\bot}$
\be\label{eq:lower2.1}
I_7(i)\ieg A_7\sqrt{N}e^{\frac{\wh S-2\eta}{\e}}
\ee
where $\eta>0$ is given by the definition of the path $\g_0$.
\end{lemma}

\begin{proof}
We have to be careful with the change of variable. Let us write the Jacobian matrix $J\g(s,y_{\bot})$ in the local base $(\g'_0(s),\g'_0(s)^{\bot})$, if we denote $P_0$ the projection on $Span(\g'_0(s))$, we get the Jacobian matrix (written by blocks)
\be\label{eq:lower2.2}
J\g(s,y_{\bot})
=
\begin{pmatrix}
1+P_0(A'(s)y_{\bot})&0\\
*&A(s)
\end{pmatrix}
\ee
since $\mathrm{Im}A(s)=\g'_0(s)^{\bot}$. 
Then, as $A(s)$ is an isometry, we obtain that
\be\label{eq:lower2.3}
g_i(s,y_{\bot})=\abs{\det(J\g(s,y_{\bot}))}=\abs{1+P_0(A'(s)y_{\bot})}=1+O(\de).
\ee
Thus, for $\de$ sufficiently small, 
\begin{align}\label{eq:lower2.4}\nonumber
I_7(i)
&=
\int_{\abs{s}>\de_0}e^{S^{N}(s,y_{\bot})/\e}g_i(s,y_{\bot})^{-1}\ds
\ieg(1+C\de)e^{\frac{\wh S-2\eta}{\e}}(s_++s_-)
\ieg 2(s_++s_-)e^{\frac{\wh S-2\eta}{\e}}
\end{align}
since $S^{N}(s,y_{\bot})<\wh S-2\eta$ for all $\abs{s}>\de_0$, and $y_{\bot}\in C^{N,\bot}_{\de}$. Then by construction of the path we have that  
\be\label{eq:lower2.5}
s_++s_-\ieg C\norm{x_{i-}-x_{i+}}_{2,N}\ieg C\sqrt{N}\norm{x_{i-}-x_{i+}}_{L^2}.
\ee
\end{proof}

We insert \eqref{eq:lower1.1} and \eqref{eq:lower2.1} in Equation \eqref{eq:lower.10}. Then  we proceed as in the proof of Lemma 4.7 from \cite{bbm10} and we obtain
\begin{align}\label{eq:lower.7b}
I_5(i)
&\seg \e \sqrt{\frac{\abs{\la_{N,0}}}{2\pi\e}}
\int_{B^{N,\bot}_{\de}(z^*_i)}(a^{i+}(y_{\bot})-a^{i-}(y_{\bot}))^2e^{-\frac{S^N(z^*_i+(0,y_{\bot}))}{\e}}\dy_{\bot}
\cro{1+A_6\frac{\de_0^{3}}{\e} +A'_7e^{-\frac{\eta}{\e}}}^{-1}.
\end{align}
Using Equation \eqref{eq:initial.2} from Proposition \ref{prop:initial}, we obtain for all $y_{\bot}$, $\abs{a^{j}(y_{\bot})-a^{j}(0)}<e^{-\frac {C}{\e}}$. Then using the approximation \eqref{eq:approx.low.2} and following the proof of Lemma 4.7 in \cite{bbm10}, we obtain
%
for $\de=\sqrt{K\e\abs{\ln{\e}}}$ and $\de_0=K'\e\abs{\ln\e}$ with $K'>K$,
\be\label{eq:lower.11b}
I_5(i)
\seg
\e (a^{i-}-a^{i+})^2e^{-\frac{S^N(z^*_i)}{\e}}
\frac{\sqrt{2\pi\e}^{N-2}\abs{\la_{N,0}}}{\sqrt{|\det(H S^N(z^*_i))|}}(1-A_5\sqrt{\e}\abs{\ln(\e)}^{3/2}).
\ee
Equation \eqref{eq:lower.1} follows by minimizing along the $(a^j)_j$.
\end{proof}

\subsection{Uniform estimate of  the mass of the equilibrium potential}\label{ssec:equilibrium}

We prove estimates of the numerator of \eqref{eq:estimate.2}. 
Let us denote $x_{l_0}^*\in\bbR^N$ to be the closest minimum to $\phi_{l_0}$ in $L^2([0,1])$.
We will prove an adaptation of Proposition 4.9 of \cite{bbm10}.
\begin{proposition}\label{prop:num}
For all $\e<\e_0$ and $\rho$, we have
\be\label{eq:num.1}
\int_{\bbR^N}h^*(x)\dint \mu^N(x)
= \frac{(2\pi\e)^N}{\sqrt{\det{HS^N(x_{l_0}^*)}}}e^{-\frac{{S^N}(x_{l_0}^*)}{\e}}(1+\psi_2(\e,N))
\ee
where $\abs{\psi_2(\e,N)}<\sqrt{\e}\abs{\ln(\e)}^{3/2}$ for all $N>N_0$.
\end{proposition}

\begin{proof}
As the previous section, we define around the minimum $x_{l_0}^*\in\bbR^N$ a neighborhood $C^N_{\de}(x_{l_0}^*)$. 
In the local orthonormal basis of the minimum $x^*_{l_0}$, the neighborhood $C^N_{\de}(x_{l_0}^*)$ is defined by
\be\label{eq:approx.1b}
C^N_{\de}(x_{l_0}^*)=\acc{y\in
\bbR^N:\,|y_l|\ieg \de\frac{r_{l}}{\sqrt{\abs{\la_{N,l}}}},\, 0\ieg
l\ieg N-1}+x_{l_0}^*
\ee
where $(r_l)$ is a sequence satifying $\sum_{l}\frac{r_l^{3/2}}{l^{3/2}}<\infty$ and $(\la_{N,l})_l$ are the eigenvalues in the increasing order of $HS^N(x_{l_0}^*)$.

We need to estimate 
\be\label{eq:num.0}
\int_{\bbR^N}h^*(x)\dint \mu^N(x).
\ee
Let us remark that for $x\in\del C^N_{\de}(x^*)$, then one of the coordinate is precisely $\de r_k/\sqrt{\la_{k,N}}$ thus
\be\label{eq:num.3b}
S^N(x)>S^N(x^*)+\de^2r_k^2-C\de^3>S^N(x^*)+c\de^2.
\ee
We consider $S'$ such that the set $\acc{\phi, S(\phi)\in]S(\phi_{l_0}),S']}$ contains no stationary point. Then using Proposition \ref{prop:conv.potentiel}, for all $\eta$ small enough, there exists $N_0$ such that for $N>N_0$, $\acc{x, S^N(x)\in[S^N(x^*)+\frac12 c\de^2,S'-\eta]}$ contains no stationary point. 
We define the set $A=\acc{S^N(x)\ieg S^N(x^*)+c\de^2}\setminus \cB_{\rho}^N(x^*)$.
Note also that for $\de$ small enough, $C^N_{\de}(x^*)\subset \cB_{\rho}^N(x^*)$.
Hence we decompose \eqref{eq:num.0} in three parts:
\be\label{eq:num.5a}
\int_{\bbR^N}h^*(x)\dint \mu^N(x)
=I_8+\int_{S^N(x)\seg S^N(x_{l_0}^*)+c\de^2}h^*(x)\dint \mu^N(x)+\int_{A}h^*(x)\dint \mu^N(x)
\ee
To estimate the third integral we need a control on the equilibrium potential on the set $A$. 

\begin{lemma}\label{lem:equilibrium}
For all $\rho<\rho_0$ and $\eta>0$ there exists $\e_0(\rho)$ such that for $\e<\e_0$ and $\de>0$, let $x\in A$, we have
\be\label{eq:equilibrium.1}
h_N^*(x)
=\bbP_x[\tau^N_{\e}(B^N_{\rho}(x^*))<\tau^N_{\e}(B^N_{\rho}(\cM^N_l))]
\ieg e^{-(S'-S^N(x)-2\eta)/\e}.
\ee
\end{lemma}

\begin{proof}
By definition of the set $A$ all the paths from $x\in A$ to $x^*$ attain a height of $S'-\eta$ at least. To prove this fact, let us take a path from $x$ to $x^*$, it must attain its maximum $\wh S$ at some time $t_0$. This maximum must satisfies $\wh S>S^N(x^*)+c\de^2$, since if it is not the case then from Equation \eqref{eq:num.3b}, the path must stay in $C^N_{\de}(x^*)$ which contradicts the fact that $x$ is in $A$. Then the minimal path from $x$ to $x^*$ must attain its maximum at a stationary point of height greater than $S^N(x^*)+c\de^2$ thus of height greater than $S'-\eta$. 
This gives us an easy lower bound for the rate function on the set of transition from $x\in A$ to $x^*$.
Then using the method from \cite{freidlin.wentzell84} and the uniform large deviation principle, we prove that
\be\label{eq:equilibrium.3}
h^*(x)=
\bbP_x[\tau^N_{\e}(\cB^N_{\rho}(x^*))<\tau^N_{\e}(\cB^N_{\rho}(\cM^N_l))]
\ieg e^{-(S'-2\eta- S^N(x))/\e}
\ee
uniformly in $N$.
\end{proof}

We get from \eqref{eq:num.5a}
\begin{align}\label{eq:num.5}
\int_{\bbR^N}h^*(x)\dint \mu^N(x)
&\ieg I_8+\int_{S^N(x)\seg S^N(x_{l_0}^*)+c\de^2}e^{-S^N(x)/\e}\dx
+\int_{S^N(x)\ieg S^N(x_{l_0}^*)+c\de^2}e^{-(S'-2\eta)/\e}\dx
\end{align}
where we have used the fact that $h^*$ is bounded by one for the second integral and the previous lemma for the third integral.
The integral $I_8$ gives the main contribution and is estimated as in the proof of Proposition 4.9 of \cite{bbm10} using the quadratic approximation of the potential on $C^N_{\rho}(x_{l_0}^*)$. The second integral on the right-hand side is estimated as in the proof of Lemma 4.6 in \cite{bbm10}. 

We bound the third integral by the volume of the set $\acc{S^N(x)\ieg S^N(x_{l_0}^*)+c\de^2}$ which is bounded uniformly in $N$. In fact, from the bound on $S^N$ and the convergence of $S^N(x_{l_0}^*)$ to $S(\phi_{l_0})$, we get for $\de$ sufficiently small
\be\label{eq:num.7}
\acc{S^N(x)\ieg S^N(x_{l_0}^*)+c\de^2}
\subset
\acc{\norm{\grad^Nx}_{2,N}^2+\norm{x}_{2,N}^2<N(S(\phi_{l_0})+c)}
\ee
which is a deformed ball. The computation shows that this quantity is uniformly bounded in $N$.

We obtain the result since the order of magnitude of the two last integrals ($O\pare{e^{-(S'-\eta)/\e}}$) of \eqref{eq:num.5} is much smaller than $I_8=O\pare{e^{-S^N(x^*_{l_0})/\e}}$.
\end{proof}

\subsection{Finite Dimensional Formula}\label{ssec:finite.time}

The finite dimensional Formula is now obtained with a uniform control in the dimension.
From Proposition \ref{prop:stationary}, we take $x^*=\phi^N_{l_0}$ where $\phi^N_{l_0}$ is the unique minimum of $S^N$ such that
\begin{align}\label{eq:stationary.1}
\norm{\phi_{l_0}-\phi_{l_0}^N}_{L^2}&\ieg \frac{C}{N}
&
\norm{\wh\phi_{l_0}^N-\phi_{l_0}^N}_{\infty}&\ieg \frac{C}{\sqrt{N}}
\end{align}
where $\wh\phi_{l_0}^N$ is the linear interpolation of $\phi_{l_0}$.
\begin{proposition}\label{prop:expect}
Let $\tau_{\e}^N$ be the transition time from $\cB_{\rho}^N(\phi^N_{l_0})$ to $\cB_{\rho}^N(\cM^N_l)$, we have uniformly in $N$
\be\label{eq:expect.1}
\bbE_{\phi^N_{l_0}}\cro{\tau_\e^N}
=\frac{2\pi e^{\frac{{S^N}(\phi^N_{l_0})}{\e}}}{C^*(N,\e)\sqrt{\det{HS^N(\phi^N_{l_0})}}}(1+\Psi(\e,N))
\ee
where $C^*(N,\e)$ is the equivalent conductance and
\be\label{eq:expect.2}
\limsup_{N\to+\infty}\abs{\Psi(\e,N)}\ieg C\sqrt{\e}\abs{\ln{\e}}^{3/2}.
\ee
\end{proposition}

\begin{proof}
Inserting the estimates for the capacity (Proposition \ref{prop:capa}) and the numerator (Proposition \ref{prop:num}) in Equation \eqref{eq:estimate.2} we conclude that
\be\label{eq:estimate.3b}
\bbE_{\nu^{N}}[\tau^N_{\e}]
=\frac{2\pi e^{\frac{{S^N}(\phi^N_{l_0})}{\e}}}{C^*(N,\e)\sqrt{\det{HS^N(\phi^N_{l_0})}}}(1+\Psi_1(\e,N))
\ee
where $\limsup_{N}\abs{\Psi_1(\e,N)}<C\sqrt{\e}\abs{\ln(\e)}^{3/2}$ and $\nu^{N}$ is a probability measure on $\del \cB^N_{\rho}(\phi^N_{l_0})$. 
Now we use Proposition \ref{prop:initial} to replace the measure $\nu^{N}$ by the point $\phi^N_{l_0}$.
For $y\in \cB^N_{\rho}(\phi^N_{l_0})$, we have by definition
\begin{align}
\norm{\phi^N_{l_0}-y}_{L^2}^2&<\rho^2&
\abs{S^N(\phi^N_{l_0})-S^N(y)}&<\rho.
\end{align}
Then from Proposition \ref{prop:stationary}, we have $N_0$ such that for $N\seg N_0$ 
\begin{align}
\norm{\phi_{l_0}-y}_{L^2}^2&<2\rho^2&
\abs{S(\phi_{l_0})-S^N(y)}&<2\rho.
\end{align}
Thus since $V$ is regular, we obtain
$\abs{\norm{\phi'_{l_0}}^2_{L^2}-\norm{y'}^2_{L^2}}<C\rho$.

Let $z=y-\phi_{l_0}$, we have by integration by parts
\begin{align}
\abs{\norm{\phi'_{l_0}+z'}^2_{L^2}-\norm{\phi'_{l_0}}^2_{L^2}}
&=\abs{2\bra{\phi'_{l_0},z'}+\norm{z'}^2_{L^2}}
=\abs{-2\bra{\phi''_{l_0},z}+\norm{z'}^2_{L^2}}<C\rho
\end{align}
since $\phi_{l_0}$ is regular as a classical solution of a differential equation. Then we obtain by the Cauchy-Schwarz inequality
\begin{align}
\norm{z'}^2_{L^2}
\ieg C\rho +2\norm{\phi''_{l_0}}_{L^2}\norm{z}_{L^2}
\ieg (C+2\norm{\phi''_{l_0}}_{L^2})\rho.
\end{align}
Thus we get
\be
\norm{y-\phi^N_{l_0}}_{\infty}\ieg \norm{y-\phi_{l_0}}_{\infty}\ieg C'\norm{y-\phi_{l_0}}_{H^1}= C'\norm{z}_{H^1}\ieg C''\sqrt{\rho}.
\ee
Using Proposition \ref{prop:initial}, we get that for all $N\seg N_0$
\be
\abs{\bbE_{\nu^N}\cro{\tau_{\e}^{N}}-\bbE_{{\phi_{l_0}}^N}\cro{\tau_{\e}^{N}}}
\ieg e^{\frac{\wh S-2\eta}{\e}}
\ee
which gives us \eqref{eq:expect.1} since the exponential asymptotics of \eqref{eq:estimate.3b} is greater than $e^{\frac{\wh S-\eta}{\e}}$.
\end{proof}

\subsection{Proof of Theorem \ref{th:main}}

From Proposition \ref{prop:expect} applied to the finite diffusion approximation where the minima and saddle points are given by Proposition \ref{prop:stationary}, we have
\be
\bbE_{\phi^N_{l_0}}\cro{\tau_\e^N}
=\frac{2\pi h_N e^{\frac{{S^N}(\phi^N_{l_0})}{\e}}}{C^*(N,\e)\sqrt{\det{HS^N(\phi^N_{l_0})}}}(1+\Psi(\e,N))
\ee
where the factor $h_N$ comes from the time change (Equation \eqref{eq:finite.2b}). Using Proposition \ref{prop:conv.eigen} (convergence of the eigenvalues) and Corollary \ref{cor:conv.eigen2} (convergence of the ratio of eigenvalues), the quadratic forms $Q^N$ converges to $Q$:
\begin{align}\label{eq:quad.1b}\nonumber
\frac1{h_N}Q^N(a)\sqrt{\det{HS^N(\phi^N_{l_0})}}
&=\sum_{{\phi^*}^N_l}\frac{\abs{\la_N^-({\phi^*}^N_l)}}{h_N}\sqrt{\frac{\det{HS^N(\phi^N_{l_0})}}{\abs{\det{HS^N({\phi^*}^N_l)}}}}e^{-\frac{S^N({\phi^*}^N_l)}{\e}}
(a_{l+}-a_{l-})^2\\\nonumber
\frac1{h_N}Q^N(a)\sqrt{\det{HS^N(\phi^N_{l_0})}}
&\conv{N}{+\infty}
\sum_{{\phi^*}_l}\abs{\la^-({\phi^*}_l)}
\sqrt{\frac{\Det{\cH_{\phi_{l_0}} S}}{\abs{\Det{\cH_{\phi^*_l} S}}}}e^{-\frac{S({\phi^*}_l)}{\e}}(a_{l+}-a_{l-})^2\\
&\qquad\qquad=Q(a)e^{-\frac{S({\phi^*}_l)}{\e}}\sqrt{\Det{\cH_{\phi_{l_0}} S}}.
\end{align}
where ${\phi^*}^N_l$ are the relevant saddle points given by Proposition \ref{prop:stationary}.
Then the minimizer converges. For all $\e$, we get
\be
\frac{1}{h_N}C^*(N,\e)\sqrt{\det{HS^N(\phi^N_{l_0})}}
\conv{N}{\infty}
\cC^*(\phi_{l_0},\cM_{l})e^{-\frac{S({\phi^*}_l)}{\e}}\sqrt{\Det{\cH_{\phi_{l_0}} S}}.
\ee
Therefore, we obtain the result of Theorem \ref{th:main} from Proposition \ref{prop:conv.time}.

\end{document}